%% file: main.tex
\documentclass{article}
\DeclareUnicodeCharacter{FF0C}{,}

\usepackage[margin=1in]{geometry}  

\usepackage{listings}
\usepackage{amsmath,mathtools}
\usepackage{amsthm}
\usepackage{tikz}
\usepackage{caption,subcaption}
\usepackage{array}
\usepackage{mdwmath}
\usepackage{multirow}
\usepackage{mdwtab}
\usepackage{eqparbox}
\usepackage{multicol}
\usepackage{amsfonts}
\usepackage{multirow,bigstrut,threeparttable}
\usepackage{amsthm}
\usepackage{array}
\usepackage{bbm}
\usepackage{epstopdf}
\usepackage{mdwmath}
\usepackage{mdwtab}
\usepackage{eqparbox}
\usepackage{latexsym}
\usepackage{amssymb}
\usepackage{bm}
\usepackage{amssymb}
\usepackage{graphicx}

\usepackage{mathrsfs}
\usepackage{epsfig}
\usepackage{psfrag}
\usepackage{setspace}
\usepackage{tikz-cd}

\usepackage{tcolorbox}

\usepackage[
CJKbookmarks=true,
bookmarksnumbered=true,
bookmarksopen=true,
colorlinks=true,
citecolor=red,
linkcolor=blue,
anchorcolor=red,
urlcolor=blue
]{hyperref}
\usepackage{cleveref}
\usepackage[ruled]{algorithm2e}
\usepackage{stfloats}

\usepackage[backend=biber, style=numeric, giveninits=true, sorting=nyt]{biblatex} 

\DeclareNameAlias{author}{family-given}
\usepackage{nameref}
\usepackage[normalem]{ulem}

\usepackage{soul}

\input xy
\xyoption{all}

\theoremstyle{plain}
\newtheorem{theorem}{Theorem}[section]
\newtheorem{proposition}{Proposition}[section]
\newtheorem{lemma}{Lemma}[section]

\newtheorem{assumption}{Assumption}[section]

\crefname{theorem}{Theorem}{Theorems}
\crefname{proposition}{Proposition}{Propositions}
\crefname{corollary}{Corollary}{Corollaries}

\theoremstyle{definition}
\newtheorem{definition}{Definition}[section]
\newtheorem{remark}{Remark}[section]

\crefname{definition}{Definition}{Definitions}
\crefname{remark}{Remark}{Remarks}

\def\EE{\mathbb{E}}

\def\PP{P}
\def\QQ{Q}
\def\RR{\mathbb{R}}

\def\calB{\mathcal{B}}
\def\calC{\mathcal{C}}

\def\calI{\mathcal{I}}

\def\calK{\mathcal{K}}

\def\calN{\mathcal{N}}

\def\calP{\mathcal{P}}

\def\calS{\mathcal{S}}

\def\calX{\mathcal{X}}
\def\calY{\mathcal{Y}}
\def\calZ{\mathcal{Z}}

\newcommand{\cc}{\textup{c}}

\newcommand{\Let}{\triangleq}

\def\1{\mathbbm{1}}
\def\var{\mathsf{Var}}

\usepackage{booktabs}
\usepackage{adjustbox}
\usepackage{multirow}
\usepackage{listings}

\theoremstyle{plain}

\def \var {\mathsf{Var}}

\usepackage{xspace}

\newcommand{\diff}{{\rm d}}

\definecolor{myblue}{rgb}{.8, .8, 1}
\definecolor{mathblue}{rgb}{0.2472, 0.24, 0.6} 
\definecolor{mathred}{rgb}{0.6, 0.24, 0.442893}
\definecolor{mathyellow}{rgb}{0.6, 0.547014, 0.24}

\usepackage{enumitem}

\usepackage{accents}

\input{sty}
\usepackage{dsfont}
\usepackage{authblk} 

\usepackage{pgfplots}
\pgfplotsset{compat=1.17}
\usepackage[T1]{fontenc} 

\newcommand{\Cwass}{{CW}}
\renewcommand{\Let}{\triangleq}
\usepackage{wrapfig}

\title{Statistical Inference in Causal Partial Identification \\ with Smooth Densities}

\author[1]{Sirui Lin}
\author[2]{Zijun Gao\thanks{The first two authors contributed equally to this work.}}
\author[1]{Jos\'{e} Blanchet}
\author[1]{Peter Glynn\thanks{Emails: \texttt{zijungao@marshall.usc.edu,  siruilin, jose.blanchet, glynn@stanford.edu}}}

\affil[1]{Department of Management Science and Engineering, Stanford University}
\affil[2]{Marshall School of Business, University of Southern California}

\begin{document}
\maketitle
\begin{abstract}
Many causal quantities are only partially identifiable due to the inherent missingness of potential outcomes, and the associated partial identification (PI) sets can be obtained by solving an optimal transport (OT) problem.
Covariates often provide additional information about the potential outcomes and thus yield tighter PI sets, which can be obtained via conditional optimal transport (COT).
However, COT-based PI set estimators are susceptible to the curse of dimensionality in the covariates and outcomes, which precludes the asymptotic normality and hinders statistical inference.
In this paper, we exploit smoothness in the marginal densities of covariates and potential outcomes and develop a wavelet-based primal method for COT with multivariate outcomes and covariates.
Moreover, for quadratic cost functions, we establish a stability result for COT and prove asymptotic normality of the proposed estimator. This characterization of the asymptotic distribution enables valid statistical inference for the partial identification set.
Empirically, we validate the estimation and inference performance of our approach through numerical experiments in comparison with existing benchmarks.
\end{abstract}


\textbf{Keywords:} Causal inference; partial identification; conditional optimal transport; curse of dimensionality; smooth density estimation; wavelet methods.
\section{Introduction}\label{sec:introduction}

We study statistical inference for causal estimands that are only partially identifiable under the potential outcome framework. 
The target is a partial identification (PI) set, which can be characterized through an optimal transport (OT) formulation~\parencite{gao2024bridging}. 
When covariates are available, conditional optimal transport (COT) yields tighter PI sets by conditioning on the covariate distribution, thus provides a natural framework for inference in modern observational and experimental datasets~\parencite{ji2023model,lin2025estimation,fan2025partial}. 

The problem is important for causal inference because PI sets quantify uncertainty caused by missing counterfactual outcomes, and their widths directly inform decisions in policy evaluation, fairness assessments, and treatment effect analysis. 
In particular, PI sets can be interpreted as sharp bounds on counterfactual distributions or causal estimands, which makes them central to counterfactual inference.
At the same time, COT is of independent interest in OT theory as a conditional analogue of Wasserstein distances, closely related to adapted Wasserstein metrics studied in stochastic and dynamic settings~\parencite{backhoff2022estimating}. 
Thus, advances in COT estimation and inference have immediate implications for both causal inference and statistical OT.

This perspective connects causal inference with statistical OT, an area with rapidly developing theory and methods~\parencite{villani2009optimal, panaretos2019statistical,niles2022minimax,HutterRigollet2021MinimaxOT,manole2024plugin,hundrieser2022unifying}.
This momentum reflects OT's broad impact across econometrics~\parencite{galichon2018optimal}, image analysis and signal processing~\parencite{rubner2000earth,basu2014detecting,kolouri2017optimal}, generative modeling~\parencite{arjovsky2017wasserstein,genevay2018learning,gulrajani2017improved,wang2023efficient}, distributional robustness~\parencite{mohajerin2018data,kuhn2019wasserstein,blanchet2019quantifying}, and finance/model risk~\parencite{xu2020cot,nguyen2021robustifying,backhoff2022estimating}.
Against this broader backdrop, we focus on statistical inference for causal PI sets via conditional optimal transport.

Current approaches to COT-based PI sets face substantial limitations. 
Dual formulations can yield valid confidence sets but may be conservative and require strong, hard-to-verify assumptions on potential functions~\parencite{ji2023model,al2025error}. 
Penalization and discretization approaches reduce COT to OT at the cost of conservative or slow convergence rates and an acute curse of dimensionality in the covariates and outcomes~\parencite{manupriya2024consistent,lin2025tightening,lin2025estimation}. 
Quantile regression can provide asymptotically exact confidence sets but is limited to scalar outcomes~\parencite{fan2010sharp}. 
Sampling-based methods learn transport maps or samplers rather than the optimal objective value that defines the PI set~\parencite{hosseini2025conditional}. 
Overall, existing methods do not provide asymptotically normal estimators with practical inference for multivariate outcomes in high-dimensional covariate settings.

We address these challenges with a smoothness-based, primal COT estimator built on wavelet expansions of the observable marginal densities, inspired by recent advances in statistical OT~\parencite{HutterRigollet2021MinimaxOT,niles2022minimax,manole2024plugin}. 
Unlike dual approaches that require smoothness or regularity of Kantorovich potentials and involve solving optimization problems over function spaces to recover the dual potentials~\parencite{ji2023model,al2025error}, our primal estimator instead exploits smoothness of the densities, an assumption that is more directly verifiable in practice since the densities are observable but potentials are not. 
The resulting estimator simply plugs in a smooth density estimate to compute the COT value, without the need to solve any optimization problem.
Under smoothness assumptions, we establish faster convergence rates for general Lipschitz costs, and for quadratic costs we prove a COT stability result that leads to asymptotic normality and valid confidence sets.

Our contributions span both causal PI inference and COT theory. On the causal side, we provide a principled inferential procedure with improved rates and asymptotically normal confidence sets; to our knowledge, this is the first PI inference approach via COT that delivers asymptotic normality for multivariate outcomes under smoothness. 
On the OT side, we establish a stability-based central limit theorem for conditional transport under quadratic costs, extending smooth OT results to the conditional setting; to our knowledge, this is the first such CLT for conditional optimal transport, as existing OT CLTs do not cover conditional transport values. 
A direct comparison of convergence rates and assumptions with prior COT estimators appears in \Cref{tab:cot_rates}. 
The following points summarize the methodological, theoretical, and empirical contributions.


    \begin{enumerate}
        \item We develop a primal, smoothness-based estimator for COT using wavelet expansions of observable marginal densities, complementing dual approaches that rely on potential smoothness~\parencite{ji2023model,al2025error}. Under smoothness conditions, we establish convergence rates for general Lipschitz costs, demonstrating fast estimation accuracy for causal PI sets.
 
        \item For quadratic costs, we generalize OT stability to the COT setting and leverage it to prove asymptotic normality and conduct valid statistical inference for causal PI sets. 
        \item We provide empirical evidence that the proposed estimator improves the estimation accuracy over existing COT benchmarks and asymptotically controls the type I error for inference.
    \end{enumerate}

\paragraph{Organization.}
\Cref{sec:problem_form} introduces the potential outcomes framework, conditional optimal transport, and wavelet estimation. 
\Cref{sec:estimation} presents the wavelet-based COT estimator and its convergence rates. 
\Cref{sec:inference} focuses on quadratic costs and develops the COT stability result and asymptotic normality. 
\Cref{sec:empirical} reports experimental results. 
\Cref{sec:discussion} concludes with limitations and future directions. 
All proofs and additional empirical results appear in the supplementary materials.

\paragraph{Notation.}
We denote $[n] = \{1,...,n\}, n\wedge m = \min(n,m), n \vee m = \max(n,m)$. 
We denote by $\calP(\Omega)$ the set of probability measures on $\Omega$. 
We use $\mathcal{P}_{\mathrm{ac}}(\Omega)$, for $\Omega \subseteq \mathbb{R}^d$, to denote the set of probability measures that are absolutely continuous with respect to the Lebesgue measure on $\mathbb{R}^d$.
We use $\mu_{Y,Z}$ to denote the probability distribution (and its associated measure on $\calY \times \calZ$) of $(Y, Z)$ under $\mu$, with $\mu_{Z}$ denoting the marginal distribution of $Z$, and $\mu^{z}_{Y}$ the conditional distribution of $Y$ given $Z=z$. By $\mu = \mu_Z \otimes \mu_Y^{Z}$, we mean that for any measurable function $g$, $\int g \diff \mu = \int_{\calZ} \int_{\calY} g(z,y) \diff \mu_Y^{z}(y) \diff \mu_Z(z)$. 
We denote the set of the joint couplings $\pi$ with marginals $\mu, \nu \in \calP(\calX)$ ($\calX$ may be $\calY, \calZ$ or $\calY \times \calZ$) by
\begin{equation}\label{eq:coupling_set}
    \Pi(\mu, \nu) \Let \left\{\pi \in \calP(\calX^2): \pi_{X} = \mu,~\pi_{X'} = \nu \right\}.
\end{equation}
Given a density function $f(y,z)$, we denote its marginal density of $z$ by $f_Z(z)$, and the conditional density of $y$ given $z$ by
$f(y \mid z) = {f(y,z)}/{f_Z (z)}$.
Given a measurable objective function $h$ and probability distributions $\PP, \QQ \in \calP(\RR^{d_X})$, the optimal transport distance is defined as 
    \[
    W_h(\PP, \QQ) \Let \min_{\pi \in \Pi(\PP, \QQ)} \EE_{\pi}[h(X, X')],
    \]
    When $h(x,x') = \|x - x'\|_2^p$, we denote $W_p(\PP, \QQ) = W_h(\PP, \QQ)^{\frac{1}{p}}$, which is the so-called Wasserstein $p$-distance.

For joint distributions $\PP$ and $\QQ$ on $\mathcal Y \times \mathcal Z$,
we write $\phi_Z$ and $\psi_Z$ for the Kantorovich potentials associated
with the conditional laws $\PP_Y^z$ and $\QQ_Y^z$.
We use $P_Z$ to denote the marginal distribution of $Z$ induced from $P$, and similarly for $Q_Z$.
We use $W_p(\PP,\QQ)$ for the Wasserstein-$p$ distance between $\PP$ and $\QQ$, and $\Cwass_p(\PP,\QQ)$ for the conditional Wasserstein-$p$ distance defined as $(\EE[W_p(\PP(Y\mid Z),\QQ(Y\mid Z))^p])^{1/p}$.

\section{Preliminary}\label{sec:problem_form}

\subsection{Potential Outcome Model and Conditional Optimal Transport}\label{sec:potential.outcome.model}
Suppose there are $2n$ units, and each unit is associated with two potential outcomes \( Y_i(0) \) and \( Y_i(1) \in [0,1]^{d_Y} \) (see, e.g., \parencite{rubin1974estimating}). Specifically, \( Y_i(0) \), \( Y_i(1) \) is the outcome under control, treatment, respectively.
We allow $d_Y \ge 1$, so that the outcome, normalized to $[0,1]^{d_Y}$, may be vector-valued.
This setting naturally arises in applications such as clinical trials with primary and secondary outcomes,
physics experiments where outcomes correspond to two- or three-dimensional spatial locations,
and policy evaluations involving multiple aspects, such as effectiveness and cost.
In addition, we suppose each unit also comes with a covariate vector \( Z_i \in [0,1]^{d_Z} \), \( d_Z \ge 1 \).
We consider the following super-population model~\parencite{imbens2015causal} for $(Y_i(0), Y_i(1), Z_i)$,
\begin{align}\label{eq:super.population.model}
    (Y_i(0), Y_i(1), Z_i)
    \stackrel{\mathrm{i.i.d.}}{\sim} \mu,
\end{align}
where $\mu$ denotes an unknown joint distribution.
This super-population model satisfies the standard stable unit treatment value assumption (SUTVA) \parencite{imbens2015causal}, as described in the first paragraph of Section~\ref{sec:potential.outcome.model}.

Each unit receives a binary treatment assignment \( W_i \in \{0,1\} \), where \( W_i = 1 \) indicates treatment and \( W_i = 0 \) indicates control, and only the potential outcome corresponding to the received treatment level is observed, i.e., \( Y_i=Y_i(W_i) \).
For the treatment assignment, we focus on a completely randomized design with half of the units assigned to the treatment, that is,
\begin{align}\label{eq:CRD.half.treated}
    (W_i)_{i=1}^{2n}\ \sim\ \mathrm{Uni}\Big\{ w\in\{0,1\}^{2n}:\ \sum_{i=1}^{2n} w_i=n\Big\}.
\end{align}
Our approach and analysis can be generalized to assignment mechanisms satisfying the unconfoundedness and the overlap assumptions in~\cite{imbens2015causal}.
Under conditions \eqref{eq:super.population.model} and \eqref{eq:CRD.half.treated}, the outcome-covariate distributions of \( (Y_i(1), Z_i) \) and \( (Y_i(0), Z_i) \) are identifiable from the observed data. 
We use $P$ to denote the distribution of $(Y_i(0), Z)$ and $Q$ to denote the distribution of $(Y_i(1), Z)$, and write $p$ and $q$ for their densities when they exist. 
At the sample level, we introduce the notation $(Y_i(w), Z_i(w)) := (Y_i(w), Z_i)$ if $W_i = w$ for $w = 0,1$. Reindexing the sample, we will denote the sample from treatment group as $\calS_{1} := ((Y_i(1), Z_i(1)), i=1,...,n)$ and sample from control group as  $\calS_{0} :=((Y_i(0), Z_i(0)), i=1,...,n)$, such that $\calS_1$ can be viewed as an i.i.d.~sample from $\QQ$ and $\calS_0$ can be viewed as an i.i.d.~sample from $\PP$, both with size $n$.

In this paper, we focus on causal estimands of the form 
\begin{align}\label{eq:causal.estimand}
    V \Let \mathbb{E}_{\mu}[h(Y(0), Y(1))].
\end{align} 
Here, \( h: \calY \times \calY \rightarrow \mathbb{R} \) is a pre-specified objective function.
As $Y_i(0)$, $Y_i(1)$ are never observed simultaneously, the joint distribution of \( (Y_i(0), Y_i(1), Z_i) \), i.e., $\mu$, is not identifiable.
As a result, causal estimands~\eqref{eq:causal.estimand}  may also not be identifiable from observed data, i.e., two different joint distributions can yield different values of the causal estimand while producing the same distribution over observable quantities.
This issue is commonly known as \textit{partial identification}.
For partially identifiable causal estimands, point estimation is infeasible, and the object of interest becomes the partial identification (PI) set, that is, the set of values consistent with the identifiable marginal distributions of \( (Y_i(0), Z_i(0))\) and \( (Y_i(1), Z_i(1)) \), which we denote as 
\begin{align}\label{eq:COT.coupling}
\begin{split}
    \Pi_{\cc}(\PP, \QQ) \Let &\left\{ \pi \in \mathcal{P}(\calY^2 \times \calZ): \right.\\
&~~ \left.\pi_{Y(0), Z} = \PP, \ \pi_{Y(1), Z} = \QQ \right\}.
\end{split}
\end{align}
Given a measurable function $h$, the PI set for the moment $\mathbb{E}_{\mu}[h(Y(0), Y(1))]$ is
\begin{align}\label{eq:COT.PI}
    \left\{ \mathbb{E}_\pi[h(Y(0), Y(1))] : \pi \in \Pi_{\cc}(\PP, \QQ)  \right\}
    = [V_{\cc}, \widetilde V_{\cc}],
\end{align}
where we use the fact that the coupling set $\Pi_{\cc}$ in~\eqref{eq:COT.coupling} is convex and thus the induced PI set in~\eqref{eq:COT.PI} is a convex subset of $\mathbb{R}$ and hence an interval\footnote{The interval may be unbounded, with endpoints in $\{-\infty,+\infty\}$.}. Therefore, characterizing the PI set reduces to computing its lower and upper bounds.
In particular, we focus on the lower bound
\begin{align}\label{eq:COT.lower}
    V_{\cc}
    = \min_{\pi \in \Pi_{\cc}(\PP, \QQ) } \mathbb{E}_\pi[h(Y(0), Y(1))],
\end{align}
which can be interpreted as a \textit{conditional optimal transport} problem. The upper bound $\widetilde V_{\cc}$ is obtained analogously by replacing $h$ with $-h$.


\begin{definition}[Conditional optimal transport]
    For a measurable objective function $h$ and $\PP, \QQ$ that satisfy $\PP_Z = \QQ_Z$, the conditional optimal transport between $\PP, \QQ$ with respect to $Z$ is defined as
    \[
        \Cwass_h(\PP, \QQ) = \min_{\pi \in \Pi_{\cc}(\PP, \QQ) } \mathbb{E}_\pi[h(Y(0), Y(1))].
    \]
    For $h = \|y_0 - y_1\|_2^p$, we denote $\Cwass_p(\PP, \QQ) := (\Cwass_h(\PP, \QQ))^{\frac{1}{p}}$.
\end{definition}

\subsection{Curse of Dimensionality and Wavelet Density Estimator}\label{sec:wavelet.density.estimator}
It is well known that empirical estimation of optimal transport values such as $V_{\cc}$ suffers from the curse of dimensionality of the outcome, even in the absence of covariates (e.g., \cite{fournier2015rate}). This phenomenon is intrinsic to plug-in estimators. In particular, in the primal approach, replacing $\PP$ and $\QQ$ with the corresponding empirical distributions typically yields a statistical error of order $n^{-1/d_Y}$, which deteriorates rapidly as the dimension of the outcome space increases. In the conditional optimal transport literature, existing approaches (e.g., \cite{lin2025tightening, lin2025estimation}) additionally depend on the covariate dimension, leading to even slower convergence rates of order $n^{-1/(d_Y + d_Z)}$ for estimating $V_{\cc}$. 

In this paper, we exploit smoothness properties of the underlying data-generating distributions to mitigate the curse of dimensionality. Our approach is based on a wavelet expansion of the joint densities. For this purpose, we introduce the definition of the Besov space.
\begin{definition}[Besov space]\label{defi:Besov.space}
Let $s>0$ and $p,q\in[1,\infty]$.
Let $\Psi=(\Phi,\Psi_1,\Psi_2,\ldots)$ be a wavelet system on $[0,1]^d$,
where $\Phi$ denotes the collection of scaling functions and $\Psi_j$ the collection of wavelet functions at resolution level $j$.
Let $j_0\ge 0$ denote a fixed coarse resolution level determined by the wavelet construction.

For $f\in L^p([0,1]^d)$ admitting the wavelet expansion
\begin{align*}    
&f(x)
=
\sum_{\zeta\in\Phi}\theta_\zeta\,\zeta(x)
+
\sum_{j=j_0}^\infty\sum_{\xi\in\Psi_j}\theta_\xi\,\xi(x),\\
\theta_\zeta=&\int_{[0,1]^d}\zeta(x)f(x)\,dx, ~~
\theta_\xi=\int_{[0,1]^d}\xi(x)f(x)\,dx,
\end{align*}
the Besov norm $\|f\|_{\mathcal B^s_{p,q}([0,1]^d)}$ is defined as
\[
\|(\theta_\zeta)_{\zeta\in\Phi}\|_{\ell_p}
+
\left\|
\left(
2^{\,j(s+d/2-d/p)}
\|(\theta_\xi)_{\xi\in\Psi_j}\|_{\ell_p}
\right)_{j\ge j_0}
\right\|_{\ell_q}.
\]
The Besov space $\mathcal B^s_{p,q}([0,1]^d)$ is defined as
\[
\left\{
f\in L^p([0,1]^d):
\|f\|_{\mathcal B^s_{p,q}([0,1]^d)}<\infty
\right\}.
\]
\end{definition}
The Besov norm can be extended to negative smoothness indices $s<0$ via duality:
\[
\mathcal B^{s}_{p',q'}([0,1]^d)
=
(\mathcal B^{-s}_{p,q}([0,1]^d))^*,
~~
\frac{1}{p'}+\frac{1}{p}=1,~
\frac{1}{q'}+\frac{1}{q}=1.
\]

In this paper, we choose the boundary-corrected wavelet system on $[0,1]^{d_Y + d_Z}$ and defer the details of this wavelet system to \Cref{sec:bc_wavelet}. The reason why the wavelet expansion is useful for the analysis arises from the following result.
\begin{proposition}[{\cite[Theorem 4]{niles2022minimax}}]
    Suppose that $\widehat \mu, \mu \in \calP_{\textup{ac}}([0,1]^d)$ with densities $\hat f, f \in L^2([0,1]^d)$. If $f(x) \geq \gamma^{-1}\,\forall x\in[0,1]^d$ for some $\gamma > 0$, then there is a universal constant $C_{0}>0$ such that
    \begin{align*}
        W_2(\widehat \mu, \mu) \leq C_{0} \gamma^{\frac{1}{2}}  \|\hat f - f\|_{\calB_{2,1}^{-1}([0,1]^d)}.
    \end{align*}
\end{proposition}

\begin{definition}[Wavelet density estimator]\label{defi:wavelet.density.estimator}
Let $J_n \ge j_0$ be a resolution level. Given i.i.d.~observations $(X_i)_{i=1}^n$ with density $f$, define the wavelet projection
\[
\widetilde f_{J_n}(x)
:=
\sum_{\zeta\in\Phi}\widehat \theta_\zeta\,\zeta(x)
+
\sum_{j=j_0}^{J_n}\sum_{\xi\in\Psi_j}\widehat \theta_\xi\,\xi(x),
\]
where $\widehat\theta_{\zeta}
:=
\frac{1}{n}\sum_{i=1}^n
\zeta(X_i), \widehat\theta_{\xi}
:=
\frac{1}{n}\sum_{i=1}^n
\xi(X_i).
$

The boundary-corrected wavelet density estimator is defined as
\begin{align}\label{eq:wavelet.estimator}
\widehat f_{J_n}(x)
=
\frac{\widetilde f_{J_n}(x)\,\mathbf 1\{\widetilde f_{J_n}(x)\ge 0\}}
{\int_{\widetilde f_{J_n}(x)\ge 0}\widetilde f_{J_n}(x)\,dx}.
\end{align}
Specifically, in this work, we take $J_n = \lfloor \frac{\log_2(n)}{2s + d_Y + d_Z}\rfloor$ and $j_0 \geq \lfloor\log_2(s/0.18 + 1)\rfloor + 1$ (which is standard as discussed in \cite[Section A.2]{manole2024plugin}).
\end{definition}

Additional to a wavelet expansion, we also focus on density with uniform smoothness described by the H\"older norm.
\begin{definition}[H\"older space]\label{defi:Holder.space}
Let $s > 0$, the H\"older space $C^s([0,1]^d)$ consists of
functions $f$ whose derivatives up to order $\lfloor s \rfloor$
extend continuously to $[0,1]^d$ and satisfy
\begin{align*}
&\|f\|_{C^s([0,1]^d)}
:=
\sum_{|\alpha|\le \lfloor s \rfloor}\!
\|D^\alpha f\|_\infty +
\sum_{|\alpha|=\lfloor s \rfloor}
\sup_{\substack{x,y\in[0,1]^d\\ x\ne y}}
\frac{|D^\alpha f(x)-D^\alpha f(y)|}
{\|x-y\|^{s-\lfloor s \rfloor}}
<\infty,
\end{align*}
where $
D^\alpha f
=
\frac{\partial^{|\alpha|} f}
{\partial x_1^{\alpha_1}\cdots \partial x_d^{\alpha_d}}$, $|\alpha|=\sum_{i=1}^d \alpha_i$.
\end{definition}

To conclude this section, we present the following result, showing that any smooth density function embeds into the Besov space, admitting a wavelet expansion.
\begin{proposition}[{c.f.,\cite[Lemma 25]{manole2024plugin}}]\label{prop:embedding}
    $\calC^s([0,1]^d) \subseteq \mathcal{B}^s_{\infty,\infty}([0,1]^{d})$, $\calC^s \subseteq \mathcal{B}^{s - \epsilon}_{2,2}$ for any $\epsilon > 0$.
\end{proposition}

\subsection{Assumption}
\begin{assumption}[Bounded density]\label{a:boundedness_assp}
 Assume $P, Q \in \mathcal{P}_{\mathrm{ac}}([0,1]^{d_Y + d_Z})$, and $\gamma^{-1} \leq p$, $q \leq \gamma$ for $\gamma > 0$.
\end{assumption}

\begin{assumption}[Smoothness]\label{a:density_smooth}
 Assume that, for a fixed smoothness parameter $s>0$, the densities $p, q \in\calC^{s+\epsilon}([0,1]^{d_Y + d_Z})$ for some $\epsilon > 0$.
\end{assumption}

\begin{assumption}[Lipschitz objective]\label{a:Lip_obj}
    The objective function $h$ is $L_h$-Lipschitz, i.e., $|h(y_0) - h(y_1)| \leq L_h \|y_0 - y_1\|_2\,$, $ \forall y_0$, $y_1 \in [0,1]^{d_Y}$.
\end{assumption}

Assumption~\ref{a:boundedness_assp} postulates that the density is bounded away from zero, a condition that is both standard and practically mild in nonparametric optimal transport. When this assumption is violated, the convergence rate of optimal transport estimators is known to change qualitatively; see \cite[Section~5]{niles2022minimax} for a detailed discussion.

Assumption~\ref{a:density_smooth} imposes smoothness on the density, which is central to our analysis. This regularity enables a wavelet-based representation of the density and is crucial for constructing estimators that are statistically more efficient than a direct plug-in approaches. Note that, for technical reasons, we require the density to possess slightly higher smoothness than $s$, namely $s + \epsilon$ for an arbitrarily small $\epsilon > 0$.

Assumption~\ref{a:Lip_obj} imposes the Lipschitz continuity of $h$, which is natural in optimal transport,
as it guarantees stability of the optimal transport value with respect to perturbations of the input measures; see \parencite{villani2003topics}.

\section{Main Result}\label{sec:estimation}
\subsection{Wavelet-based COT Estimator}\label{sec:estimator_form}
In this section, we introduce the wavelet-based COT estimator, which is established upon the wavelet density estimator with alignment on the marginal density of $Z$.
\begin{definition}[Wavelet-based COT estimator]\label{defi:wavelet_estimator}
Suppose there is an i.i.d.~sample $\calS_0 = ((Y_i(0), Z_i(0)), 1\leq i \leq n)$ drawn $P$ and an i.i.d.~sample $\calS_1 = ((Y_j(1), Z_j(1)), 1\leq j \leq n)$ drawn from $Q$, $\calS_0 $ independent of $\calS_1$. Let the resolution level be $J_n = \lfloor \frac{\log_2(n)}{2s + d_Y + d_Z}\rfloor$.
    \begin{enumerate}[label=(\roman*)]
        \item Let $\widehat P_n \in \calP([0,1]^{d_Y+d_Z})$ be the distribution associated with the wavelet density estimator based on $\calS_0$ (see Definition~\ref{defi:wavelet.density.estimator}), and $\widehat Q_n$ is defined similarly based on $\calS_1$.
        \item Further, let $\widehat R_{n} = (\widehat \PP_{n,Z} + \widehat \QQ_{n,Z})/2 \in \calP([0,1]^{d_Z})$. 

        \item Let $\widehat P_n^{\dagger}(\diff y, \diff z) = {\widehat \PP_{n, Y}^z}(\diff y) \widehat R_{n}(\diff z)$, $\widehat Q_n^{\dagger}(\diff y, \diff z) = {\widehat \QQ_{n, Y}^z}(\diff y) \widehat R_{n}(\diff z)$. The wavelet-based COT estimator is defined by
        \[
            \widehat V_{\cc, n} := \Cwass_h(\widehat P_n^{\dagger}, \widehat Q_n^{\dagger}).
        \]
    \end{enumerate}
\end{definition}
Since conditional optimal transport requires alignment of the $Z$-marginal, we construct a shared $Z$-marginal using the pooled $Z$-samples from both data sources. The paired $(Y,Z)$ observations from each source are then used to estimate the corresponding conditional distributions.

Note that $\widehat P_n^{\dagger}$ and $\widehat Q_n^{\dagger}$ are continuous distributions. Consequently, $\widehat V_{\cc,n}$ is approximated via a resampling procedure, which is standard in the literature (see, e.g.,~\parencite{weed2019estimation, deb2021rates}). The detailed algorithm is deferred to \Cref{sec:optimal.transport.estimated.density}.

In contrast to \parencite{ji2023model}, our wavelet-based estimator does not require any cross-fitting and thus avoids the repeated computation across folds as well as the artificial randomness introduced by sample splitting.

\subsection{Fast Convergence Rate with General Objective}

\begin{table*}[tbp]
\centering
\caption{Comparison of statistical convergence rates for COT value estimation.}
\label{tab:cot_rates}
\begin{tabular}{lll}
\toprule
\textbf{Reference} & \textbf{Estimation Rate} & \textbf{Key Assumptions} \\
\midrule
\parencite{lin2025estimation} 
& $n^{-\frac{1}{2 \vee d_Y + d_Z}}$ 
& Lipschitz kernel \\

\parencite{ji2023model} 
& $n^{-1/2}$ 
& Density, optimal dual are estimated at rate $O(n^{-1/4})$\\

\textbf{This work}
& $n^{-\frac{s}{2s+d_Y+d_Z}}$ 
& Densities in $\mathcal C^{s+\epsilon}$, bounded away from zero  \\
\textbf{This work}
& $n^{-\frac{2s}{2s+d_Y+d_Z}} \vee n^{-1/2}$ 
& Densities in $\mathcal C^{s+\epsilon}$, bounded away from zero, $h$ quadratic\\
\bottomrule
\end{tabular}\label{tab:compare_rate}
\end{table*}

To bound the statistical convergence rate of the wavelet-based COT estimator, we first present the following stability result of COT under Lipschitz kernels and objectives.
\begin{proposition}[Stability]\label{prop:stability_I}
    Suppose that $W_1(\PP_Y^z, \PP_Y^{z'}) \leq L_p \|z - z'\|_2, W_1(\QQ_Y^z, \QQ_Y^{z'}) \leq L_p \|z - z'\|_2$ for any $z, z'$, and $h$ is $L_h$-Lipschitz, then
    \begin{align}\label{eq:stability_I}
    \begin{split}
        &|\Cwass_h(\widehat \PP, \widehat \QQ) - \Cwass_h(\PP, \QQ)| \leq 2L_h L_p W_1(\PP_Z, \widehat \PP_Z) + L_h \int W_1(\PP_Y^z, \widehat \PP_Y^z) + W_1(\QQ_Y^z, \widehat \QQ_Y^z) \,\diff  \widehat \PP_Z (z). 
        \end{split}
    \end{align}
\end{proposition}
The first term on the right-hand side of \eqref{eq:stability_I} can be bounded using standard results from wavelet density estimation \parencite{niles2022minimax}. In contrast, we derive the convergence rate of the second term for the density estimator introduced in Section~\ref{sec:estimator_form}. 
\begin{proposition}[Key estimation bound]\label{prop:core_rate_main}
Under Assumption~\ref{a:boundedness_assp}-\ref{a:density_smooth},
    \[
        \EE\left[\int W^2_2(\widehat \QQ_Y^z, \QQ_Y^z) \, \widehat R(\diff z)\right] \leq C n^{-\frac{2s}{2s+d_Y+d_Z}},
    \]
    where $C$ is a constant that depends on $P, Q, \gamma, s, d_Y, d_Z$.
\end{proposition}

We are now ready to state our main result.
\begin{theorem}[Fast convergence of COT under smooth densities]\label{thm:main_thm1}
    Under \Cref{a:boundedness_assp} ($\gamma$), \ref{a:density_smooth} ($s$) and \ref{a:Lip_obj} ($L_h$), 
    \[
        \EE[|V_{\cc} - \widehat V_{\cc, n}|] \leq C L_h n^{-\frac{s}{2s+d_Y + d_Z}},      
    \]
    where $C$ is a constant that depends on $P, Q, \gamma, s, d_Y, d_Z$.
\end{theorem}
When $s$ is sufficiently large, the convergence rate of the COT estimator is close to the parametric rate of order $n^{-1/2}$, relieving the dependence of dimension $d_Y + d_Z$. 
To conclude this section, we compare the rate and assumption of our method with other existing approach estimating the COT value that has a finite-sample convergence rate; see \Cref{tab:compare_rate}.

\subsection{Statistical Inference with Quadratic Objective}\label{sec:inference}

In this section, we focus on the quadratic objective $h(y_0, y_1) = \|y_0 - y_1\|_2^2$, which allows a stronger stability bound and thus 
an analysis of the asymptotic distribution of the wavelet-based COT estimator, enabling standard statistical inference.
The results below readily extend to any objective $h(y_0,y_1)=y_0^\top M y_1$ with $M \in \RR^{d_Y \times d_Y} \succeq 0$.

When the objective function is quadratic for an optimal transport problem, Brenier's theorem \parencite{knott1984optimal, brenier1991polar} shows that the optimal solution is a map defined as the gradient of the so-called Brenier potential function. In our conditional setting, we let $\varphi(y, z)$ be the Brenier potential between $\PP_Y^z$ and $\QQ_Y^z$ for each $z \in [0,1]^{d_Z}$. To obtain a sharper convergence-rate analysis than in the previous section, we will impose stronger smoothness assumptions on the densities $p,q$ and the functions $\varphi$. In particular, under a quadratic objective $h$, the smoothness of $\varphi$ can often be deduced from that of $p$ and $q$. Roughly speaking, for each fixed $z$, if the conditional densities $p_Y^z(y)$ and $q_Y^z(y) \in \calC^{s}([0,1]^{d_Y})$, then typically the corresponding potentials $\varphi(\cdot,z) \in \mathcal C^{s+2}([0,1]^{d_Y})$, under suitable regularity and boundary conditions. See, for instance, \cite[Chapter~12]{villani2008optimal}.

\begin{proposition}[Stability with quadratic objective]\label{prop:stability.COT}
     Under \Cref{a:boundedness_assp} ($\gamma$), and further assume that 
     \begin{enumerate}[label=(\roman*)]
         \item $\varphi(\cdot, z) \in \calC^2([0,1]^{d_Y})$ for any $z \in [0,1]^{d_Z}$.
         \item $C_2 :=\sup_{z\in[0,1]^{d_Z}}\|\varphi(\cdot, z)\|_{\calC^2([0,1]^{d_Y})}  < \infty$.
     \end{enumerate}
      Then there exists a constant $\lambda > 0$ that depends on $\gamma, C_2$ such that    
    \begin{align*}        
        0 \leq & \Cwass_2(\widehat P_n^{\dagger},\widehat Q_n^{\dagger})^2 - \Cwass_2(\PP,\QQ)^2 - \int \phi \diff (\widehat P_n^{\dagger} - \PP) 
        - \int \psi \diff (\widehat Q_n^{\dagger} - \QQ) \\
        \leq & 2\lambda \int W^2_2(\widehat \PP_Y^z, \PP_Y^z) + W^2_2(\widehat \QQ_Y^z, \QQ_Y^z) \widehat R_{n}(\diff z),
    \end{align*}
    where $\phi(y,z) = \|y\|^2_2 - \varphi(y,z)$ and $\psi(y,z) = \|y\|^2_2 - \varphi^*(y,z)$, with $\varphi^*(y,z) := \sup_{y' \in [0,1]^{d_Y}} \{y^\top y' - \varphi(y', z)\}$.
\end{proposition}

When $s$ is sufficiently large, the right-hand side of the above inequality is of order $o_p(n^{-1/2})$, as implied by Eq.~\eqref{eq:stability_I} and \Cref{thm:main_thm1}. Thus, the target converges to the distributional limit of $\int \phi \diff (\widehat P_n^{\dagger} - \PP) + \int \psi \diff (\widehat Q_n^{\dagger} - \QQ)$. 
As a result, we have the following result.
\begin{theorem}\label{theo:CLT}
Under \Cref{a:boundedness_assp} ($\gamma$) - \ref{a:density_smooth} ($s$), and that $p$, $q$, $\varphi \in \mathcal{C}^{s+1}([0,1]^{d_Y+d_Z})$, $2s > d_Y + d_Z$,
    \begin{align*}
         \sqrt{n}\left(\Cwass^2_2(\widehat \PP_n^\dagger,\widehat \QQ^\dagger_n) - \Cwass^2_2(\PP,\QQ)^2 \right) 
         \stackrel{d}{\to}
         \mathcal N(0, \sigma^2),
    \end{align*}
    where $\sigma^2 := 2 \left(\var_{\PP}(\eta(Y,Z))+\var_{\QQ}(\kappa(Y,Z))\right)$ with
    \begin{align*}
        &\eta(y,z):=\phi(y,z)
        -\frac{1}{2}\left(\int\left(\phi(y',z) \frac{p(y',z)}{p_Z(z)}-\psi(y',z) \frac{q(y',z)}{q_Z(z)}\right) \diff y'\right), \\
        &\kappa(y,z):=~\psi(y,z)
        -\frac{1}{2}\left(\int\left(\psi(y',z) \frac{q(y',z)}{q_Z(z)}-\phi(y',z) \frac{p(y',z)}{p_Z(z)}\right) \diff y'\right).
    \end{align*}
\end{theorem}

A key step in proving \Cref{theo:CLT} is to establish a central limit theorem for
\[
\int \phi \, d(\widehat P_n^{\dagger} - \PP)
+
\int \psi \, d(\widehat Q_n^{\dagger} - \QQ).
\]
Although $\widehat P_n^{\dagger}$ and $\widehat Q_n^{\dagger}$ are dependent through the shared estimator $\widehat R$ (thus the argument in \parencite{manole2023central} cannot be applied directly), the identity
$\widehat r - \widehat p = (\widehat q - \widehat p)/2$
allows a decomposition that yields the desired CLT; see \Cref{appe:sec:proof:theo:CLT}.


\begin{remark}[$d_Z = 0$]
    When there is no covariate, the central limit theorem can be derived with milder assumption. For example, for $d_Y = 1$, \cite{del2005asymptotics} shows the CLT of the plug-in estimator (using the empirical distribution) when $p,q \in \calC^1([0,1])$; \cite{manole2024plugin} extends the CLT of wavelet-based estimator for general $d_Y$ when $p, q \in \calC^s([0,1]^{d_Y}), s > d_Y/2 - 2$. By contrast, this work extends it to a more general result with covariate and identifies a similar smoothness condition $s > (d_Y + d_Z)/2$. 
\end{remark}

\begin{remark}[$\PP = \QQ$]
    The asymptotic variance $\sigma^2$ is positive if and only if the conditional distributions $P_{Y\mid Z}$ and $Q_{Y\mid Z}$ differ on a set of $z$ with positive probability.
When the two conditional distributions coincide, i.e., $P_{Y\mid Z}=Q_{Y\mid Z}$, the conditional Wasserstein distance is zero, and  
$\Cwass_2^2(\widehat P_n^\dagger,\widehat Q_n^\dagger)$ converges to zero faster than the order of $n^{-1/2}$. 
\end{remark}

\begin{remark}[Asymptotic variance estimation]
The asymptotic variance $\sigma^2$ in \Cref{theo:CLT} can be consistently estimated using plug-in estimators of the conditional potential functions. 
When these quantities are difficult to estimate, subsampling-based approaches, particularly the bootstrap, can be used to approximate the variance and construct confidence intervals.
\end{remark}

\section{Numerical Experiment}\label{sec:empirical}
The code and data are available at
\url{https://github.com/siruilin1998/causalOT_inference}.
\subsection{Estimation}\label{sec:empirical.estimation}

\begin{figure*}[tbp]
        \centering
        \begin{minipage}{0.32\textwidth}
                \centering
                \includegraphics[clip, trim = 0cm 0cm 0cm 0.75cm, width = 1\textwidth]{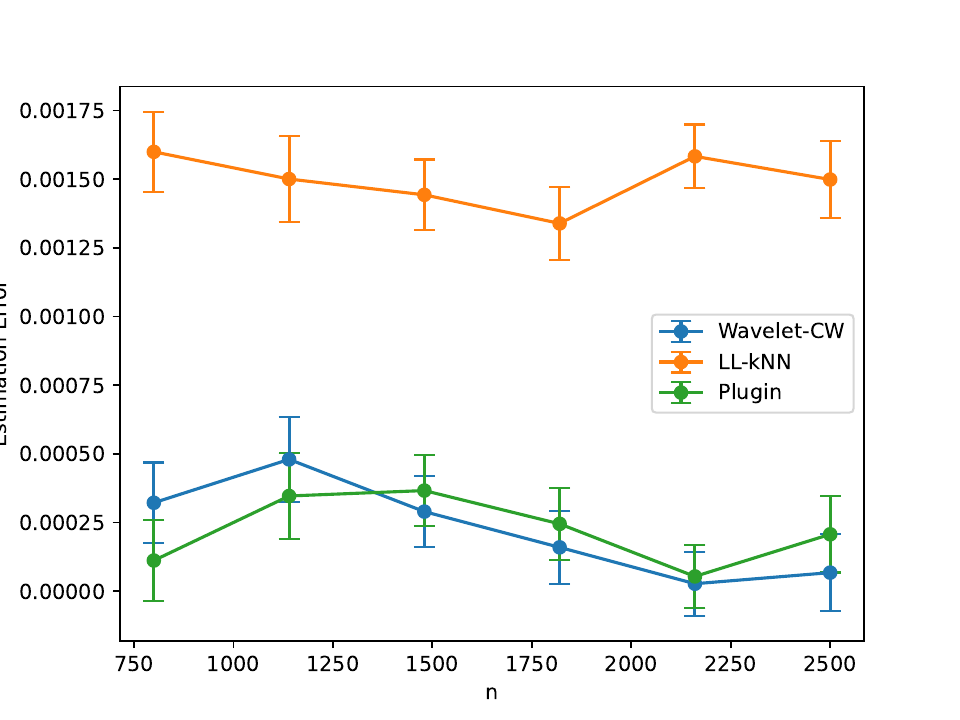}
                \subcaption{Location model}
        \end{minipage}
         \begin{minipage}{0.32\textwidth}
                \centering
                \includegraphics[clip, trim = 0cm 0cm 0cm 1cm, width = 1\textwidth]{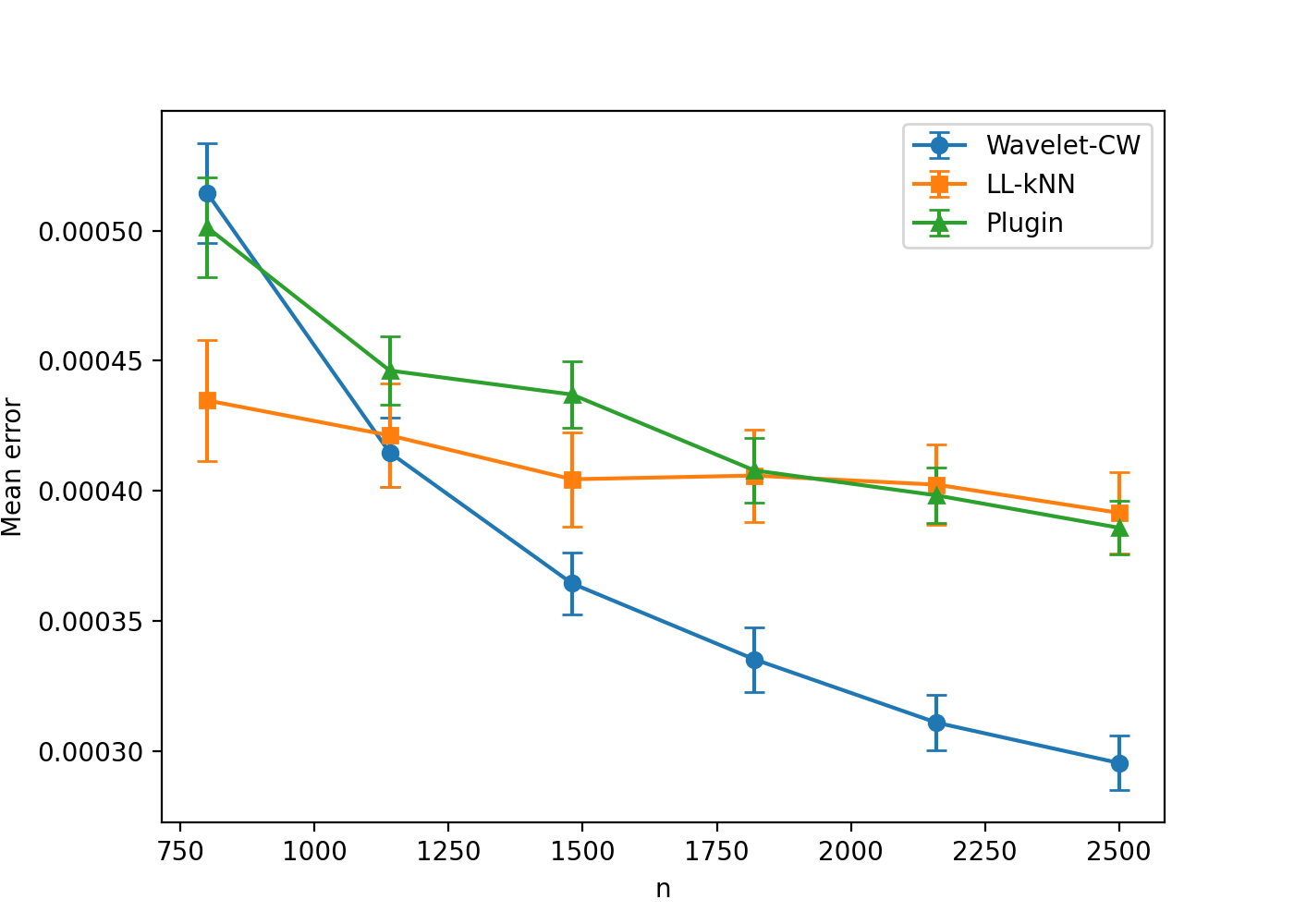}
               \subcaption{Quadratic model} 
        \end{minipage}
         \begin{minipage}{0.32\textwidth}
                \centering
                \includegraphics[clip, trim = 0cm 0cm 0cm 1cm, width = 1\textwidth]{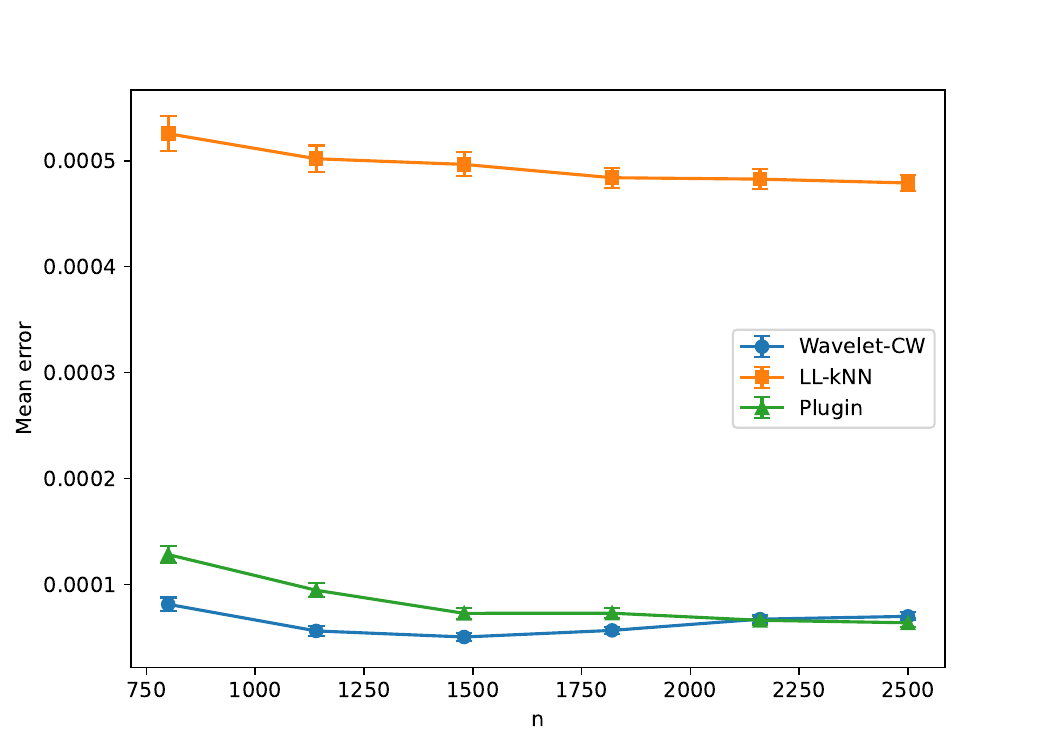}
                \subcaption{Scale model}
        \end{minipage}
        \caption{Plots of estimation error comparison of our method (\textit{Wavelet CW estimate}) with \parencite{ji2023model} (\textit{LL-kNN}, for which we use `knn''-based nuisance estimator) and \parencite{lin2025estimation} (\textit{Plugin}) with $d_Y=1$, $d_Z=1$. The estimation error is defined by $\EE[|V_{\cc} - \widehat V_{\cc}|]$. The mean error curve and the corresponding standard-error confidence bands are computed by aggregating results over 100 Monte Carlo repetitions. The plots show that the convergence rate of our method is at least comparable with the existing approaches in low dimension.
        }
        \label{fig:convergence_rate_dy1}
\end{figure*}

\begin{figure*}[h!]
        \centering
        \begin{minipage}{0.32\textwidth}
                \centering
                \includegraphics[clip, trim = 0cm 0cm 0cm 1.35cm, width = 1\textwidth]{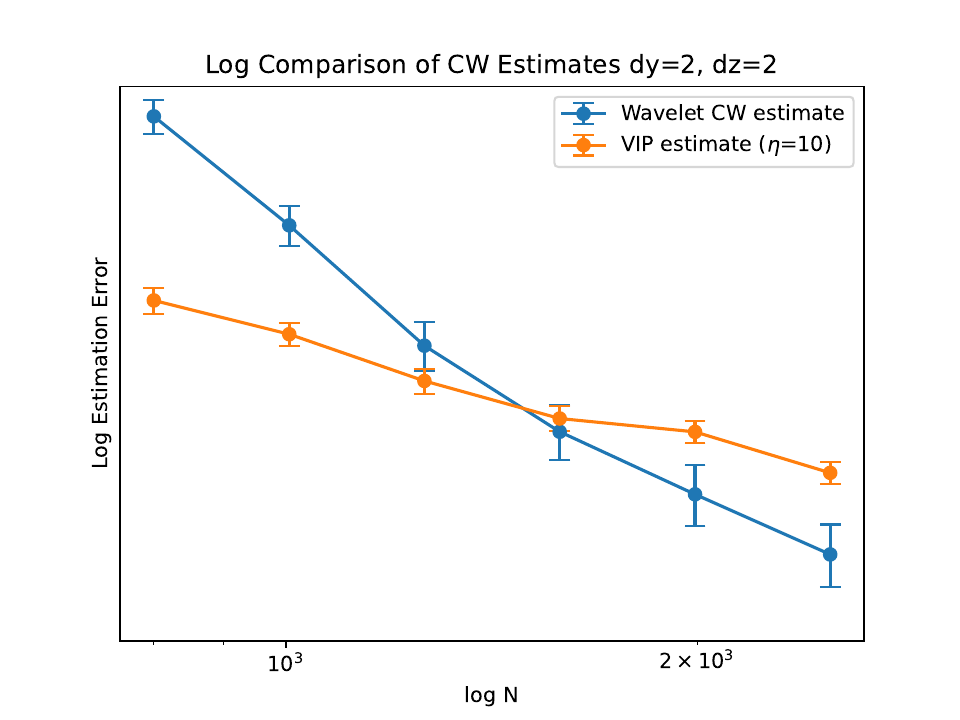}
                \subcaption{Location model}
        \end{minipage}
         \begin{minipage}{0.32\textwidth}
                \centering
                \includegraphics[clip, trim = 0cm 0cm 0cm 1.35cm, width = 1\textwidth]{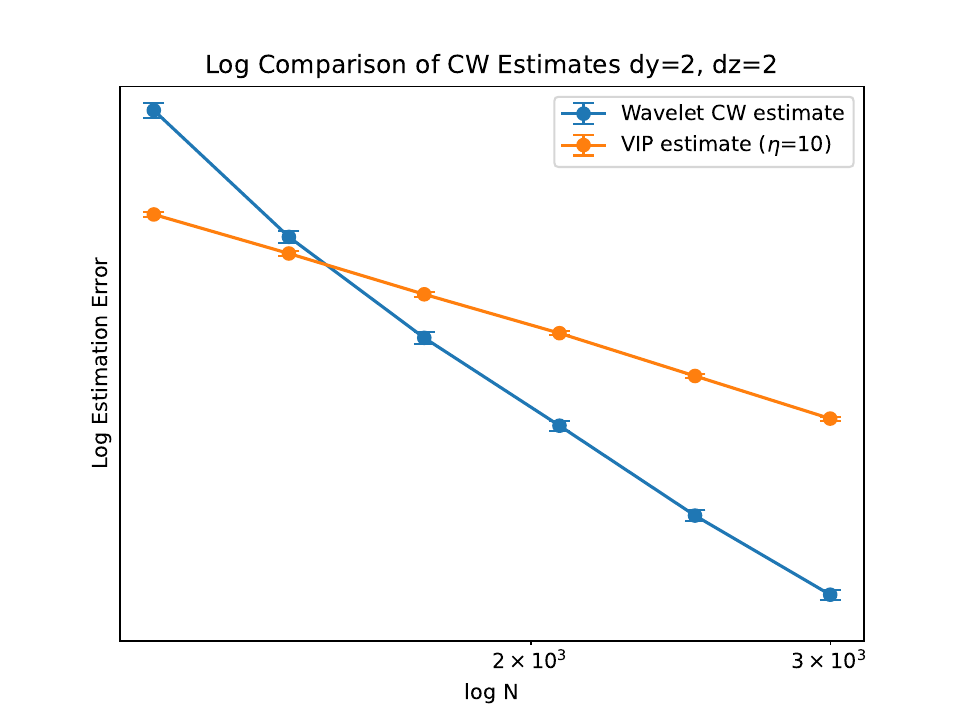}
               \subcaption{Quadratic model} 
        \end{minipage}
         \begin{minipage}{0.32\textwidth}
                \centering
                \includegraphics[clip, trim = 0cm 0cm 0cm 1.35cm, width = 1\textwidth]{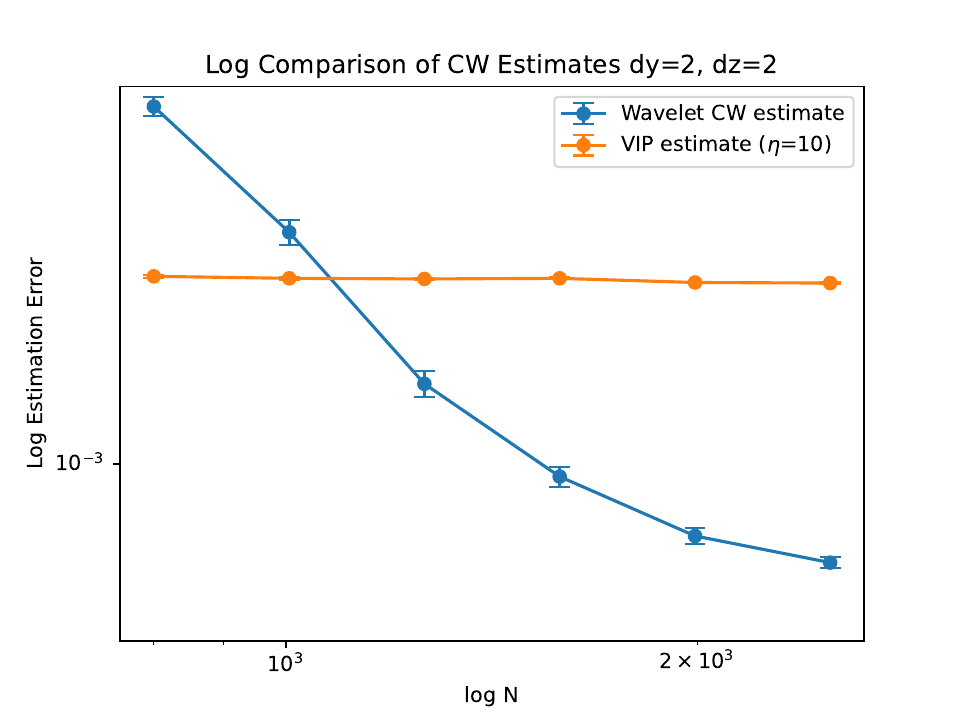}
                \subcaption{Scale model}
        \end{minipage}
        \caption{Log-log plots of estimation error comparison of our method (\textit{Wavelet CW estimate}) and \parencite{lin2025tightening} (\textit{VIP estimate} with $\eta = 10$) with $d_Y=2$, $d_Z=2$. The estimation error is defined by $\EE[|V_{\cc} - \widehat V_{\cc}|]$. The mean error curve and the corresponding standard-error confidence bands are computed by aggregating over 300 Monte Carlo repetitions. The plots show that the convergence rate of our method is significantly faster than the rate of VIP estimate, which suffers from the curse of dimensionality based on its theory.
        }
        \label{fig:convergence_rate_loglog}
\end{figure*}

We empirically compare the convergence rate of our estimator (\Cref{sec:inference}) with existing approaches including \parencite{ji2023model, lin2025tightening, lin2025estimation}.

\paragraph{Data generation mechanism.} 
We outline a default data-generating mechanism. Each dimension of the covariate $Z \in [0,1]^{d_Z}$ are drawn i.i.d.\ from $\mathrm{Uniform}(0,1)$. The potential outcomes $Y$ are drawn from
\begin{align*}  
Y(w)\mid Z=z \sim \mathcal{N}\big(\mu_w(z), \Sigma_w(z)\big), 
\quad w\in\{0,1\}.
\end{align*}
The objective function is quadratic, $h = \|y_0 - y_1\|_2^2$. This underlying distribution and objective enable a closed form evaluation of the COT value thanks to the Gelbrich formula (Proposition~\ref{prop:Gelbrich}). Note that we shall choose the parameter $\mu_w(z), \Sigma_w(z)$ appropriately such that most of its measure is supported on $[0,1]^{d_Y}$.

Specifically, we consider three data models.
\begin{enumerate}
    \item (Location model) $\mu_w(z) = \alpha_w z + \beta_w$, $\Sigma_w(z) \equiv \Sigma_w$.
    \item (Quadratic model) $\mu_w(z) = A_w (z - \alpha_w)^2 + \beta_w$, $\Sigma_w(z) \equiv \Sigma_w$, where $(z - \alpha_w)^2$ is an entry-wise square.
    \item (Scale model) $\mu_w(z) \equiv \alpha_w, \Sigma_w(z) = (\zeta_w^\top z) \times \Sigma_w$, where $(\zeta_w^\top z)$ is a scalar.
\end{enumerate}

\paragraph{Implementation.}We implement our method following Algorithm~\ref{alg:bc_wavelet_joint_density}, Algorithm~\ref{alg:example} with
the wavelet estimator parameters specified in \Cref{appe:sec:numerical.experiment}. The method \parencite{ji2023model} is implemented using \texttt{DualBounds} \footnote{\url{https://dualbounds.readthedocs.io/en/latest/index.html}}. The methods from \parencite{lin2025tightening, lin2025estimation} are implemented using the public code they upload with the paper.

\begin{table*}[tbp]
\centering
\caption{
Coverage of the proposed $95\%$ confidence intervals across different settings. 
We consider three choices of $(d_Y, d_Z)$, each evaluated at three sample sizes. 
Confidence intervals are constructed using the bootstrap with $100$ resamples. 
Reported coverage rates are aggregated over $200$ Monte Carlo repetitions.
}\label{tab:coverage}
\begin{tabular}{c|ccc|ccc|ccc}
\toprule
{Dimension} 
& \multicolumn{3}{c|}{$d_Y=2$, $d_Z=1$} 
& \multicolumn{3}{c|}{$d_Y=2$, $d_Z=2$} 
& \multicolumn{3}{c}{$d_Y=3$, $d_Z=2$} \\
{Sample size ($n$)} 
& $1000$ & $2000$ & $3000$
& $1000$ & $2000$ & $3000$
& $4000$ & $5000$ & $6000$\\
\midrule
Coverage
& 92.5\% & 93\%  &  93.5\%
& 86\% & 98.5\% & 97.5\%
& 90\% & 93.5\%  & 96\%  \\
\bottomrule
\end{tabular}
\end{table*}

\begin{figure*}[tbp]
        \centering
        \begin{minipage}{0.32\textwidth}
                \centering
                \includegraphics[clip, trim = 0cm 0cm 0cm 0.75cm, width = 1\textwidth]{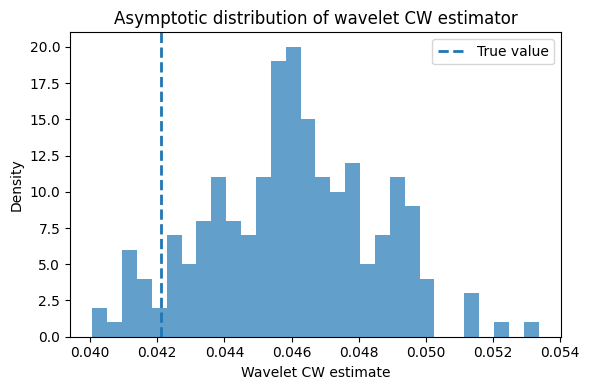}
                \subcaption{$n = 1000$}
        \end{minipage}
         \begin{minipage}{0.32\textwidth}
                \centering
                \includegraphics[clip, trim = 0cm 0cm 0cm 0.75cm, width = 1\textwidth]{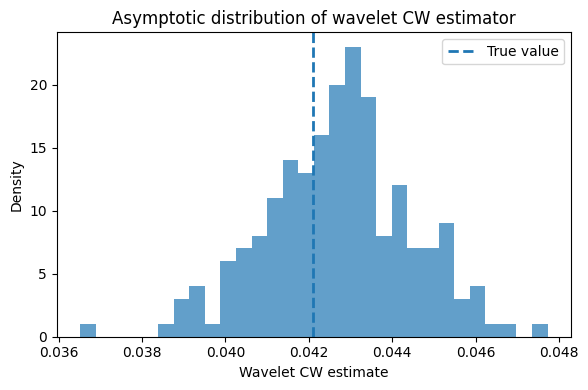}
               \subcaption{$n = 2000$} 
        \end{minipage}
         \begin{minipage}{0.32\textwidth}
                \centering
                \includegraphics[clip, trim = 0cm 0cm 0cm 0.75cm, width = 1\textwidth]{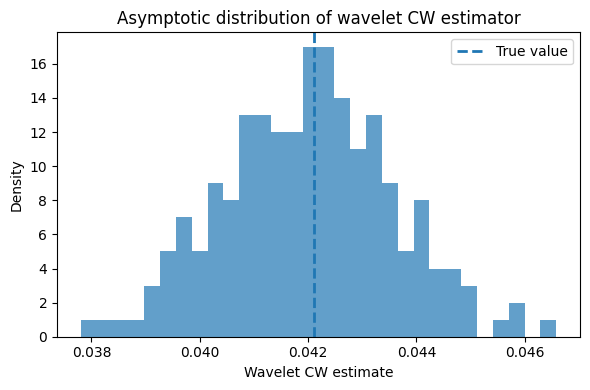}
                \subcaption{$n = 3000$}
        \end{minipage}
        \caption{Histogram of the proposed estimator with the true value marked by a vertical dashed line. 
        Here the default setting $d_Y=2$, $d_Z=2$ is adopted. The results are aggregated over $200$ Monte Carlo repetitions.
        }
        \label{fig:asymptotic.distribution}
\end{figure*}

\paragraph{Case i. ($d_Y = 1$)} We first compare our method with \parencite{ji2023model} and \parencite{lin2025estimation} in the setting of $d_Y = d_Z = 1$ across the three models. Specifically, for \parencite{ji2023model}, we use the knn-based (nonparametric) nuisance estimator to avoid model misspecification. The plots of estimation error comparison are shown in Figure~\ref{fig:convergence_rate_dy1}. The plots imply that in low-dimensional cases, our method performs comparably to existing approaches, consistent with the fact that smoothness assumptions do not improve convergence rates beyond the parametric regime.

\paragraph{Case ii. ($d_Y > 1$)} Since neither the approach of \parencite{ji2023model} nor \parencite{lin2025estimation} is implemented for $d_Y > 1$, we focus on comparing with the VIP estimate of \parencite{lin2025tightening}, for which we take $\eta = 10$ so that the estimator is asymptotically unbiased. The log-log plots of the estimation error comparison of these two methods are shown in Figure~\ref{fig:convergence_rate_loglog}. The plots demonstrate a significant improvement of convergence rate of our method comparing to the VIP estimate across multiple scenarios, which is due to the structure of our wavelet estimator that leverages the smoothness of the underlying density functions.

\subsection{Inference}\label{sec:empirical.inference}
We empirically validate the performance of inference of our proposal. In particular, we focus on the case $d_Y > 1$, where standard methods such as quantile-based approaches are not applicable.

\paragraph{Data generation mechanism.} 
We use the location model to evaluate the inference performance. Specifically, we consider a total of three scenarios with different dimensions: $(d_Y = 2, d_Z = 1)$, $(d_Y = 2, d_Z = 2)$ as the default setting, and $(d_Y = 3, d_Z = 2)$. The details of $\mu_w(z)$, $\Sigma_w(z)$ are provided in \Cref{appe:sec:numerical.experiment.simulation.details}.
For each scenario, we consider three different sample sizes $n$, and the specific values are listed in \Cref{tab:coverage}.

\paragraph{Bootstrap.} We use the bootstrap with $100$ resamples to estimate the variance $\sigma^2$ in \Cref{theo:CLT} and plug it in to construct the confidence interval.
The alternative approach for $d_Y>1$ discussed in \Cref{sec:empirical.estimation} either has non-negligible bias or lacks an established asymptotic distribution, making it unsuitable for statistical inference. We therefore focus on our method.

\paragraph{Coverage rate.}
We report the coverage results for $95\%$ confidence intervals in \Cref{tab:coverage}. 
Overall, the coverage rates are acceptable when the sample size $n$ is moderately large.
As expected, for each configuration of $(d_Y, d_Z)$, coverage improves as the sample size increases; 
as the dimensions $d_Y$ and $d_Z$ grow, larger samples are required to achieve proper coverage.

\paragraph{Asymptotic distribution.} 
In \Cref{fig:asymptotic.distribution}, we display the distribution of the proposed estimator $\Cwass^2_2(\widehat \PP_n^\dagger,\widehat \QQ^\dagger_n)$, with the true value $\Cwass_2^2(P,Q)$ marked by a vertical dashed line. 
We focus on the default setting $d_Y=2$, $d_Z=2$; histograms for the other two settings are provided in \Cref{appe:sec:numerical.experiment}. The distribution of the estimator $\Cwass_2^2(\widehat{\PP}_n^\dagger,\widehat{\QQ}_n^\dagger)$ becomes increasingly normal as $n$ grows, consistent with the CLT in \Cref{theo:CLT}. 
The bias also decreases with larger sample sizes: for $n=1000$, the estimator is upward biased, while by $n=3000$, the bias becomes negligible.

\subsection{Real Data Analysis}

\begin{wrapfigure}{r}{0.425\textwidth}
\vspace{-0.5cm}
\centering
        \begin{minipage}{0.425\textwidth}
                \centering
\includegraphics[clip, trim = 0cm 2cm 12cm 3cm, width = 1\textwidth]{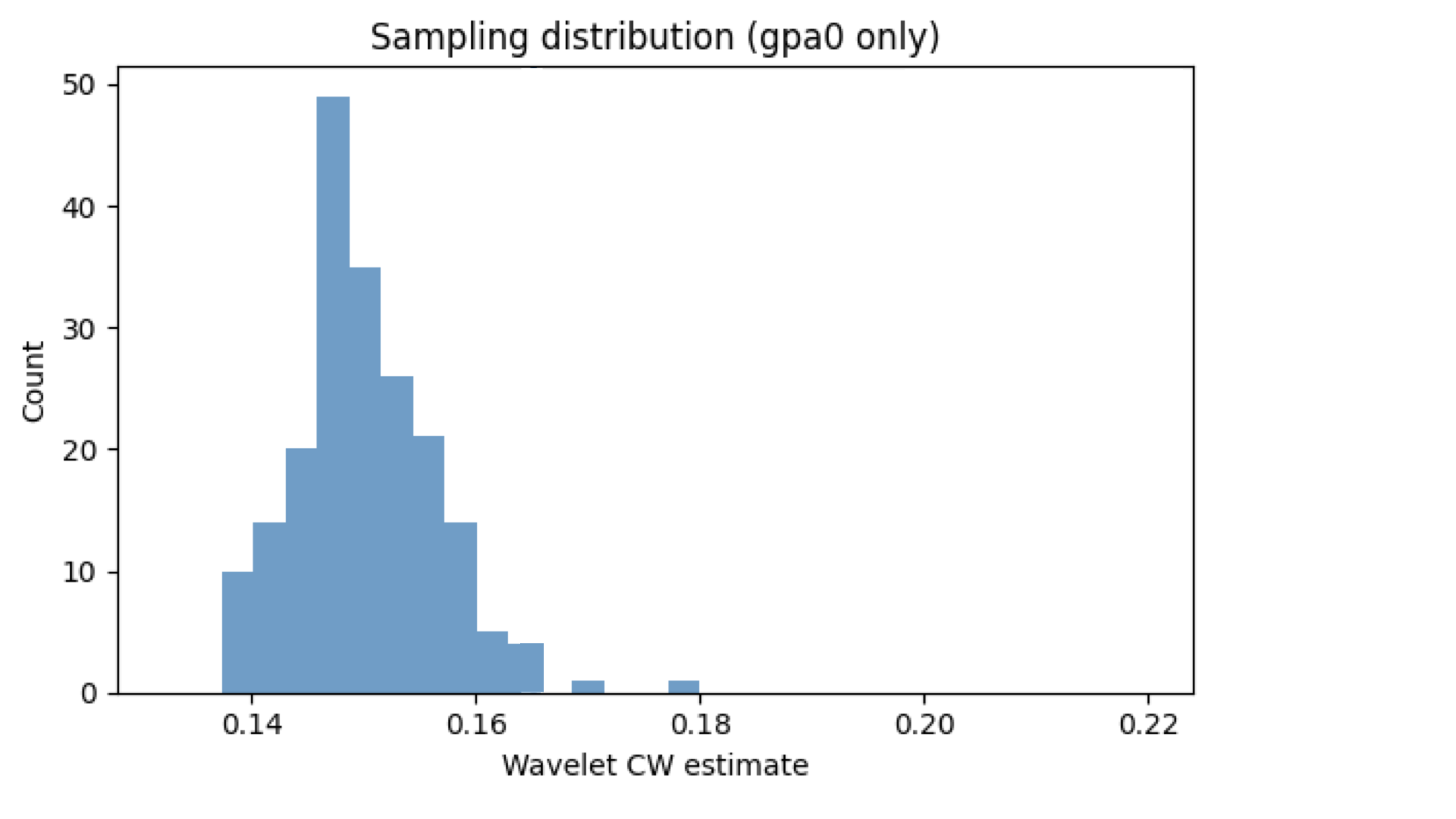}
\subcaption{Conditioning on baseline GPA only}
\vspace{0.5cm}
\end{minipage}
\begin{minipage}{0.425\textwidth}
                \centering
\includegraphics[clip, trim = 0cm 2cm 12cm 3cm, width = 1\textwidth]{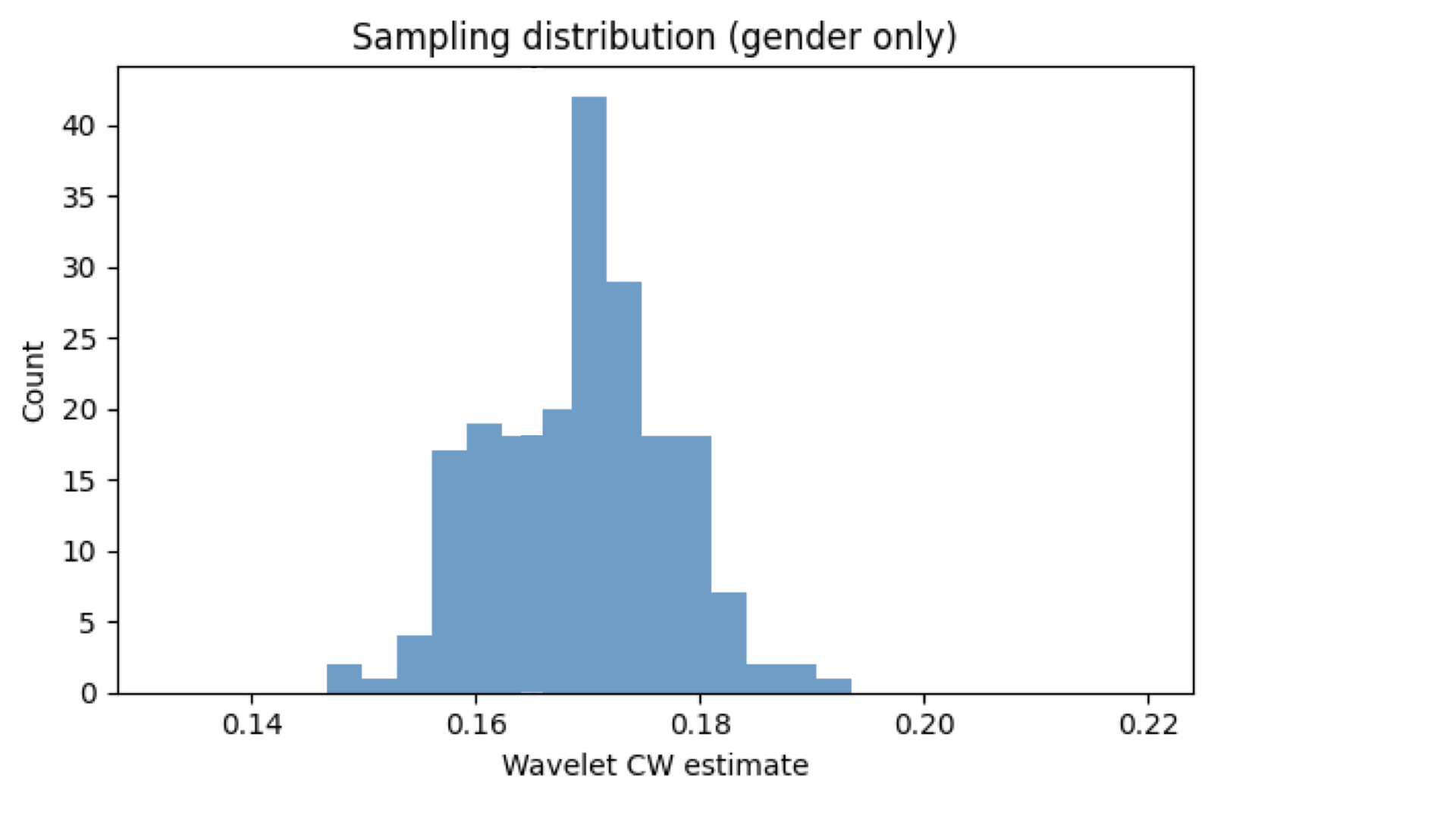}
\subcaption{Conditioning on gender only}
\vspace{0.5cm}
\end{minipage}
\begin{minipage}{0.425\textwidth}
                \centering
\includegraphics[clip, trim = 0cm 2cm 12cm 3cm, width = 1\textwidth]{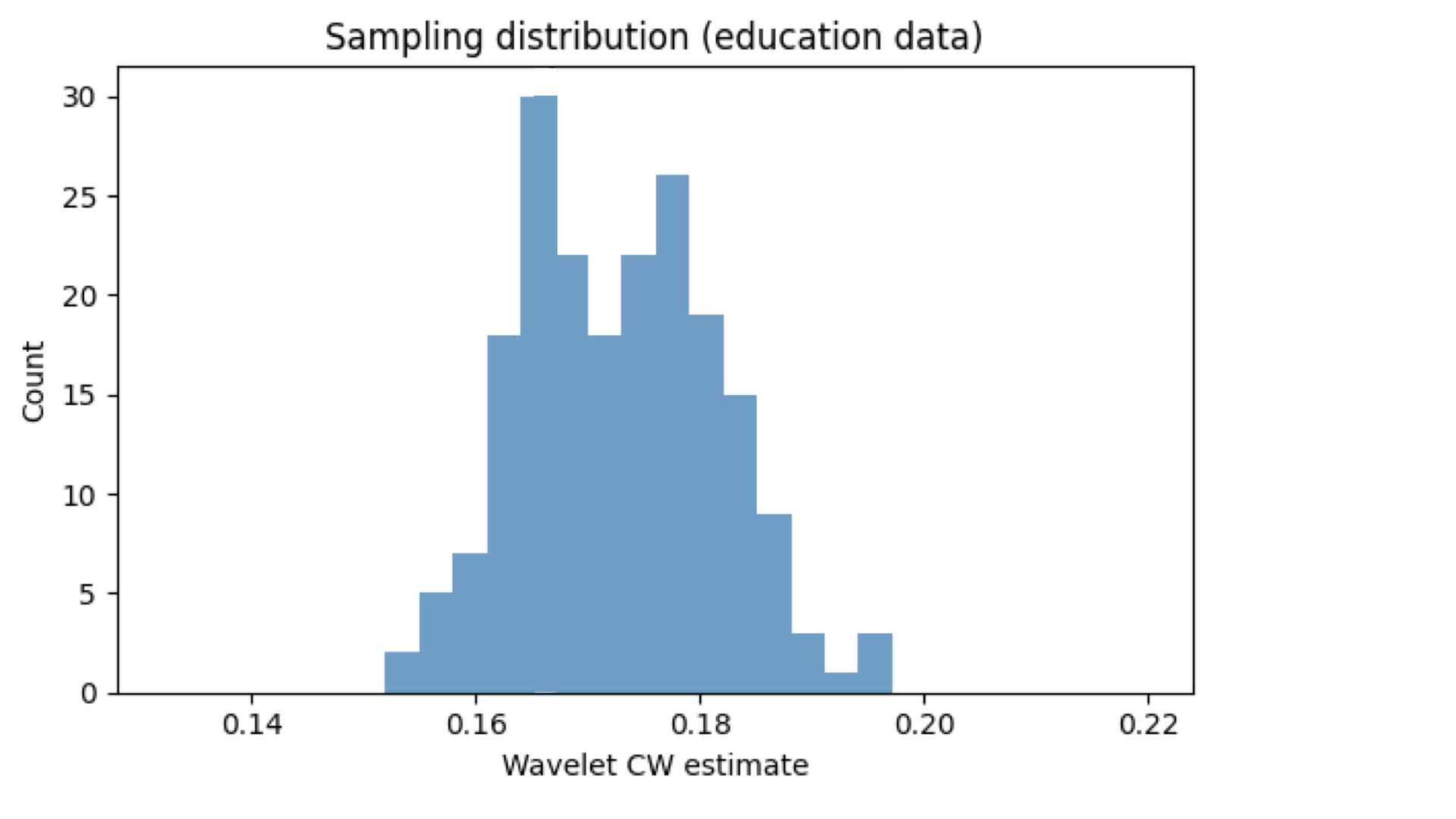}
\subcaption{Conditioning on baseline GPA and gender}
\end{minipage}
\caption{Histogram of the proposed estimator applied to hypothetical STAR datasets under different conditioning sets of covariates. Results are aggregated over $200$ simulated hypothetical datasets.}
\label{fig:education}
\vspace{-2cm}
\end{wrapfigure}

We investigate the randomized experiment from the Student Achievement and Retention Project (STAR) \parencite{angrist2009incentives}, which evaluated the effects of academic support services and financial incentives on college student performance. 
We use gender and baseline GPA (supported on $[0,4]$) as covariates, and first-year and second-year GPA (supported on $[0,4]$) as outcomes.
Because the true treatment effect is not directly observable, we construct hypothetical potential outcomes under control and treatment based on the observed outcomes. Specifically, we treat the original outcomes as control potential outcomes and generate gender-dependent treatment effects, motivated by the finding in \parencite{angrist2009incentives} that female students benefit more from the intervention. We add these treatment effects to the control outcomes to obtain treatment potential outcomes.
Based on the constructed hypothetical dataset, in each simulation replicate, we randomly sample $1300$ control and $1300$ treatment observations and use them as the observed dataset.  
We generate $200$ datasets and aggregate the results.

We apply our method to each generated dataset as described in \Cref{sec:empirical.inference}, with outcome dimension $d_Y = 2$ and covariate dimension $d_Z = 1$ or $2$.
Particularly, we consider three sets of covariates for conditioning:
\begin{itemize}
    \item [(a)] Conditioning on baseline GPA only.
    \item [(b)] Conditioning on gender only.
    \item [(c)] Conditioning on baseline GPA and gender.
\end{itemize}
By the construction of our hypothetical treatment effects, the conditional distribution under treatment differs from that under control by a location shift depending on the gender. 
Therefore, conditioning on gender, either alone or together with the baseline GPA, the optimal values of the population-wise COT recover the sharp lower bound given by the Cauchy-Schwarz inequality,
\begin{align*}
    \mathbb{E}\left[\left\|\mathbb{E}[Y(1)\mid \text{Gender}]-\mathbb{E}[Y(0)\mid \text{Gender}]\right\|_2^2\right].
\end{align*}
In contrast, when conditioning only on baseline GPA, the conditioning set does not fully capture the treatment effect heterogeneity, and therefore the optimal value of the COT is strictly smaller, leading to a looser PI set.

As shown in \Cref{fig:education}, when conditioning on gender (panels (b) and (c)), the estimators' distributions are centered at similar values, both substantially larger than when conditioning only on baseline GPA (panel (a)). This is consistent with the theoretical discussion above and highlights the importance of conditioning on informative covariates to tighten the partial identification sets. When conditioning on both gender and baseline GPA (panel (c), $d_Z=2$), the estimator exhibits slightly slower convergence, resulting in mildly larger upward bias and a less normal distribution compared to conditioning on gender alone (panel (b), $d_Z=1$). One possible future direction is to combine the proposed method with automatic selection of influential covariates.

\section{Discussion and Conclusion}\label{sec:discussion}

We propose a primal COT solver that leverages the smoothness of the marginal density to enable accurate estimation and statistical inference for causal partial identification sets. Our method accommodates multivariate outcomes, which are not supported by existing quantile-based approaches, and relieves the curse of dimensionality faced by alternative methods.
Methodologically, we establish stability results for both general and quadratic objective functions, for a procedure that incorporates a tailored alignment of the covariate marginal distribution. Empirically, our estimator converges faster than competitors and achieves desirable coverage for moderately large sample sizes.

We outline one future direction.
The direction is to extend the framework to more complex experimental designs where the probability of receiving treatment (propensity score) depends on the covariates. This extension would be useful for handling observational studies under the ignorability assumption, where treatment assignment is not randomized but can be modeled using a covariate-dependent propensity score.


\printbibliography

\appendix

\section{Optimal Transport}\label{appe:sec:OT}
\begin{definition}[Optimal transport]
    Given a measurable objective function $h$ and probability distributions $\PP, \QQ \in \calP(\RR^{d_X})$, the OT distance is defined as 
    \[
    W_h(\PP, \QQ) \Let \min_{\pi \in \Pi(\PP, \QQ)} \EE_{\pi}[h(X, X')],
    \]
    where 
    \begin{equation}
    \Pi(\PP, \QQ) \Let \left\{\pi \in \calP(\calX^2): \pi_{X} = \PP,~\pi_{X'} = \QQ \right\}.
    \end{equation}
    
    When $h(x,x') = \|x - x'\|_2^p$, we denote $W_p(\PP, \QQ) = W_h(\PP, \QQ)^{\frac{1}{p}}$, which is the so-called Wasserstein $p$-distance.
\end{definition}

The $2$-Wasserstein distance between Gaussian measures admits a closed-form expression.
\begin{proposition}[OT between Gaussian distributions \parencite{gelbrich1990}]\label{prop:Gelbrich}
    \begin{equation*}
    W_2^2\!\left(\mathcal N(m_0,\Sigma_0),\,\mathcal N(m_1,\Sigma_1)\right)
    =
    \|m_0-m_1\|_2^2
    +
    \operatorname{tr}\!\left(
    \Sigma_0+\Sigma_1
    -
    2(\Sigma_1^{1/2}\Sigma_0\Sigma_1^{1/2})^{1/2}
    \right).
    \end{equation*}
\end{proposition}

\section{Boundary-corrected Wavelet System}\label{sec:bc_wavelet}
Next, we describe the specific wavelet basis used in our proposal.
\begin{definition}[Boundary-corrected wavelet basis]
Let $j_0\ge 0$ be the minimal resolution level at which boundary correction is applied. The boundary-corrected wavelet basis on $L^2([0,1])$ is defined by
\begin{align*}
    \Phi^{\mathrm{bc}}_0 &= \{\phi^{\mathrm{bc}}_{j_0 k}, k \in \calK_0(j_0)\},\\
    \Psi^{\mathrm{bc}}_0 &= \{\psi^{\mathrm{bc}}_{j k}, k \in \calK_0(j), j \geq j_0\},
\end{align*}
where $\calK_0(j) := \{0,1,\dots,2^j-1\}$. Specifically, the wavelet basis is compactly supported orthonormal on $[0,1]$. The concrete definition of $\phi^{\mathrm{bc}}, \psi^{\mathrm{bc}}$ can be found in \cite[Section A.2.1]{manole2024plugin}.

For outcomes $y\in[0,1]^{d_Y}$ and covariates $z\in [0,1]^{d_Z}$, define
\[
x=\left((y_s)_{1\le s\le d_Y},\ (z_{s'})_{1\le s'\le d_Z}\right)
\in[0,1]^{d_Y+d_Z}.
\]
The tensor-product boundary-corrected wavelet functions on $[0,1]^{d_Y+d_Z}$ are defined by
\[
\left\{
    \begin{aligned}
    &\Phi^{\mathrm{bc}}_{j_0, \mathbf k}(x) = \prod_{r=1}^{d_Y+d_Z} \phi^{\mathrm{bc}}_{j_0 k_r}(x_r), \qquad \mathbf k \in \calK_0(j_0)^{d_Y+d_Z}\\
    &\Psi^{\mathrm{bc}}_{j,\mathbf k,\mathbf l}(x)
    =
    \prod_{r=1}^{d_Y+d_Z}
    \begin{cases}
    \phi^{\mathrm{bc}}_{j,k_r}(x_r), & l_r = 0,\\[6pt]
    \psi^{\mathrm{bc}}_{j,k_r}(x_r), & l_r = 1,
    \end{cases}
    \qquad \mathbf l \in\{0,1\}^{d_Y+d_Z} \backslash \{(0,...,0)\}, \mathbf k \in \calK_0(j)^{d_Y+d_Z}, j \geq j_0
    \end{aligned}
\right\},
\]
where $\mathbf k=(k_1,\ldots,k_{d_Y+d_Z})$ and
$\mathbf l=(l_1,\ldots,l_{d_Y+d_Z})$.
\end{definition}
\section{Additional Theoretical Results and Proofs}\label{sec:tech}
\textbf{Notation.} In the following sections, for simplicity, we will write $\widehat P_n$ (resp. $\widehat Q_n$, $\widehat R_{n}$) as $\widehat P$ (resp. $\widehat Q$, $\widehat R$). 

In this section, we present the technical proof of the core result (Proposition~\ref{prop:core_rate}). As a preparation, we introduce the following lemmas.
\begin{lemma}[{\cite[Theorem 4]{niles2022minimax}}]\label{prop:measure.and.density}
    Suppose that $\widehat \mu, \mu \in \calP_{\textup{ac}}([0,1]^d)$, and let $\hat f, f \in L^2([0,1]^d)$ be the density function of $\widehat \mu, \mu$, respectively. If there exists a constant $\gamma >0$ such that $\hat{f}(x)\vee f(x) \geq \gamma^{-1}\,\forall x\in[0,1]^d$, then 
    \begin{align*}\label{eq:Wass_B21}
        W_2(\widehat \mu, \mu) \leq C_{\textup{wb}} \gamma^{\frac{1}{2}}  \|\hat f - f\|_{\calB_{2,1}^{-1}([0,1]^d)},
    \end{align*}
    where $C_{\textup{wb}}>0$ is a universal constant.
\end{lemma}

\begin{lemma}[{\cite[Section 6.10, Proposition 7]{meyer1992wavelets}}]\label{lem:wavelet_oper}
    There is a universal constant $C_{\textup{emb}} > 0$, such that for any $f \in L^2([0,1]^d)$, 
    \begin{align*}
        \|f\|_{\calB_{2,1}^{-1}([0,1]^d)} \leq C_{\textup{emb}} \|f\|_{L_2([0,1]^d)}.
    \end{align*}
\end{lemma}

\begin{lemma}[{\cite[Lemma 29]{manole2024plugin}}]\label{lem:density_bound}
    Assume there exist $\gamma, s > 0$ such that a density function
    $q(x) \ge \gamma^{-1}$, $\forall x \in [0,1]^d$, and such that
    $q \in \mathcal{B}^{s}_{\infty,\infty}([0,1]^d)$. Let $\tilde q_n$ be the wavelet projection based on an i.i.d.~sample of size $n$
    drawn from the distribution with density $q$.
    Then, there exists $c_{\textup{b}} > 0$, depending on $[0,1]^d$ and
    $\|q\|_{\mathcal{B}^{s}_{\infty,\infty}([0,1]^d)}$, such that with probability at least
    $1 - c_{\textup{b}}/n^2$, 
    \[
        \sup_{x \in [0,1]^d} |q(x) - \tilde q_n(x)| \leq \gamma^{-1}/2.
    \]
\end{lemma}

\begin{proposition}[Key estimation bound, Proposition~\ref{prop:core_rate_main}]\label{prop:core_rate}
Under Assumption~\ref{a:boundedness_assp}-\ref{a:density_smooth},
    \[
        \EE\left[\int W^2_2(\widehat \QQ_Y^z, \QQ_Y^z) \, \widehat R(\diff z)\right] \leq C n^{-\frac{2s}{2s+d_Y+d_Z}},
    \]
    where $C = C_{\textup{wb}}^2 \gamma^2(16 \gamma^5  (C_{\textup{emb}} \gamma)^2 (C_{\textup{emb}} + 1)C^\dagger + c_{\textup{b}})$, and $C^\dagger ,c_{\textup{b}}$ depend on $d_Y, d_Z, q$.
\end{proposition}

\begin{proof}[Proof of Proposition~\ref{prop:core_rate}]
    Denote the (wavelet) density of $Q$ (resp. $P$) by $q$ (resp. $p$), the (wavelet) density of $\widehat Q$ (resp. $\widehat P$) by $\hat q$ (resp. $\hat p$), and the (wavelet) density of $\widehat R$ by $\hat r$. As for the wavelet projection, we replace $\hat p$ with $\tilde p$. Then, by \Cref{a:boundedness_assp}(i),
    \[
        q(y|z) = \frac{q(y,z)}{q_Z(z)} = \frac{q(y,z)}{\int q(y,z) \diff y} \geq \gamma^{-2}.
    \]    
    
    Then, applying \Cref{prop:measure.and.density}, we get
    \[
        \EE\left[\int W^2_2(\widehat \QQ_Y^z, \QQ_Y^z) \, \widehat R(\diff z)\right] \leq C_{\textup{wb}}^2 \gamma^2 \EE\left[\int_{\calZ} \|\hat q_Y^z(y) - q_Y^z(y)\|^2_{\calB_{2,1}^{-1}([0,1]^{d_Y})} \diff \widehat R(z)\right].
    \]

    To further investigate the conditional densities, we apply Lemma~\ref{lem:density_bound} under \Cref{a:density_smooth} and get, there exists $c_{\textup{b}} > 0$, depending on $d_Y, d_Z, q$, such that with probability at least $1 - c_{\textup{b}} / n^2$, 
    \begin{align*}
        \sup_{(y,z) \in [0,1]^{d_Y + d_Z}} |q(y, z) - \tilde q(y, z)| &\leq \gamma^{-1}/2,\\
        \sup_{(y,z) \in [0,1]^{d_Y + d_Z}} |p(y, z) - \tilde p(y, z)| &\leq \gamma^{-1}/2.
    \end{align*}
    As a result,
    \begin{align}
        \min_{(y,z) \in [0,1]^{d_Y + d_Z}} \tilde q(y,z) &\geq \gamma^{-1}/2, \label{eq:minmax_1}\\
        \tilde r(z) = \frac{1}{2}( \tilde p_Z(z) + \tilde q_Z(z)) & \in [\gamma^{-1}/2, 2\gamma] \ \  \forall z \in [0,1]^{d_Z}. \label{eq:minmax_2}
    \end{align}
    Then, on this high probability event, the wavelet projection $\tilde q$ is equal to the wavelet density estimator $\hat q$ (so does $\hat r$), thus Eq.~\eqref{eq:minmax_1}-\eqref{eq:minmax_2} also hold for $\hat q, \hat r$.
    
    Also, on this high probability event, we apply the following decomposition to bound the right-hand side of the above inequality.    
    \begin{align*}
    &\int_{[0,1]^{d_Z}} \|\hat q_Y^z(y) - q_Y^z(y)\|^2_{\calB_{2,1}^{-1}([0,1]^{d_Y})} \diff \widehat R(z)\\
    = & \int_{[0,1]^{d_Z}} \left\|\frac{\hat q(\cdot, z)}{\hat q_Z(z)} - \frac{q(\cdot, z)}{q_Z(z)}\right\|^2_{\calB_{2,1}^{-1}([0,1]^{d_Y})} \diff \widehat R(z) \quad (\text{by the definition of}~\hat q_Y^z(y))\\
    \leq & 2 \int_{[0,1]^{d_Z}} \frac{1}{\hat q_Z(z)^2} \left\|\hat q(\cdot, z) - q(\cdot, z)\right\|^2_{\calB_{2,1}^{-1}([0,1]^{d_Y})} + \|q(\cdot, z)\|^2_{\calB_{2,1}^{-1}([0,1]^{d_Y})} \left(\frac{1}{\hat q_Z(z)} - \frac{1}{q_Z(z)}\right)^2  \diff \widehat R(z)\\
    \leq & 2 (2\gamma)^3  \underbrace{\int_{[0,1]^{d_Z}} \left\|\hat q(\cdot, z) - q(\cdot, z)\right\|^2_{\calB_{2,1}^{-1}([0,1]^{d_Y})}  \diff z}_{\text{(Term A)}} +  2 (2\gamma)^3 \gamma^2 (C_{\textup{emb}} \gamma)^2\underbrace{  \int_{[0,1]^{d_Z}}  \left(\hat q_Z(z) - q_Z(z)\right)^2  \diff z}_{\text{(Term B)}},
    \end{align*}
    where the last inequality is due to \Cref{a:boundedness_assp} and Eq.~\eqref{eq:minmax_1}-\eqref{eq:minmax_2}. Specifically, by Lemma~\ref{lem:wavelet_oper}, 
    \[
        \|q(\cdot, z)\|^2_{\calB_{2,1}^{-1}([0,1]^{d_Y})} \leq C_{\textup{emb}}^2 \|q(\cdot, z)\|^2_{L^2([0,1]^{d_Y})} \leq C_{\textup{emb}}^2 \gamma^2 \qquad \forall z \in [0,1]^{d_Z}.
    \]

    We analyze the two terms separately.

    \medskip
    \noindent
    \textbf{(Term A).}
    By Lemma~\ref{lem:wavelet_oper}, 
    \begin{align*}
        &\int_{[0,1]^{d_Z}} \left\|\hat q(\cdot, z) - q(\cdot, z)\right\|^2_{\calB_{2,1}^{-1}([0,1]^{d_Y})} \diff z \\
        \leq& C_{\textup{emb}}\int_{[0,1]^{d_Z}} \|\hat q(\cdot, z) - q(\cdot, z)\|^2_{L_2([0,1]^{d_Y})} \diff z\\
        \leq & C_{\textup{emb}} \|\hat q - q\|^2_{L_2([0,1]^{d_Y + d_Z})}.
    \end{align*}

    
    \medskip
    \noindent
    \textbf{(Term B).}
    Let $\delta(y,z)=\hat q(y,z)-q(y,z)$. Then, by the Cauchy–Schwarz inequality,
    \begin{align*}
        \int_{[0,1]^{d_Z}}  \left(\hat q_Z(z) - q_Z(z)\right)^2  \diff z
        = & \int_{[0,1]^{d_Z}} \Big(\int_{[0,1]^{d_Y}} \delta(y,z) \diff y\Big)^2 dz\\
    \le & \int_{[0,1]^{d_Z}} \Big(\int_{[0,1]^{d_Y}} 1^2 \diff y\Big)\Big(\int_{[0,1]^{d_Y}} \delta(y,z)^2\diff y\Big)\diff z\\
    \le & \int_{[0,1]^{d_Y + d_Z}} (\hat q(y,z)-q(y,z))^2  \diff y \diff z,\\
    =& \|\hat q - q\|^2_{L_2([0,1]^{d_Y + d_Z})}.
    \end{align*}
    

    Note that Eq.~\eqref{eq:minmax_1}-\eqref{eq:minmax_2} implies $\hat q = \tilde q$, thus $\hat q$ is the wavelet (linear) projection of the form
    \[
        \hat q = \tilde q = \sum_{\zeta\in\Phi}\widehat \theta_\zeta\,\zeta
+
\sum_{j=j_0}^{J_n}\sum_{\xi\in\Psi_j}\widehat \theta_\xi\,\xi.
    \]

    To bound $\EE\left[\|\hat q - q\|^2_{L_2([0,1]^{d_Y + d_Z})}\right]$ , we introduce the following result
    \begin{proposition}\label{prop:L2-bound}
    Under \Cref{a:boundedness_assp}-\ref{a:density_smooth}, there exists a constant $C^\dagger$ depends on $q, \gamma, s$ such that
    \[
        \EE\left[\int_{[0,1]^{d_Y + d_Z}} (\tilde q(y,z)-q(y,z))^2  \diff y \diff z\right] \leq C^\dagger n^{-\frac{2s}{2s+d_Y+d_Z}}.
    \]
    \end{proposition}
    Also note that \parencite{hardle2012wavelets} has a similar bound for the case $d = 1$.

    Combining the above results implies the desired result. 
    \end{proof}

\subsection{Proof of Proposition~\ref{prop:stability_I}}
\begin{proof}[Proof of Proposition~\ref{prop:stability_I}]
    By \cite[Proposition 2]{lin2025estimation},
    \[
        \Cwass_h(\PP, \QQ) = \int W_h(\PP^z_Y, \QQ^z_Y) \diff  \PP_Z(z).
    \]
    Therefore, we have
    \[
        \begin{aligned}
            & \left|\Cwass_h(\PP, \QQ) - \Cwass_h(\widehat \PP, \widehat \QQ)\right|\\
            = & \left|\int W_{h}\left(\PP_{Y}^{z}, \QQ_{Y}^{z}\right) \diff \PP_{Z}(z) - \int W_{h}\left(\widehat \PP_Y^{z}, \widehat \QQ_Y^{z}\right) \diff \widehat \PP_Z(z)\right|\\ 
            \leq & \underbrace{\left|\int W_{h}\left(\PP_Y^{z}, \QQ_Y^{z}\right) \diff \PP_Z(z) - \int W_{h}\left(\PP_Y^{z}, \widehat \QQ_Y^{z}\right) \diff \widehat \PP_Z(z)\right|}_{\text{(Term A)}} +  \underbrace{\left|\int W_{h}\left(\PP_Y^{z}, \widehat \QQ_Y^{z}\right) \diff \widehat \PP_Z(z) - \int W_{h}\left(\widehat \PP_Y^{z}, \widehat \QQ_Y^{z}\right) \diff \widehat \PP_Z(z) \right|}_{\text{(Term B)}}.
        \end{aligned}
    \]
    
    For \textbf{(Term A)}: we highlight that we consider the optimal coupling $\pi_{Z, Z'} \in \calP(\calZ^2)$ that is attained in the optimal transport $W_{1}(\PP_{Z}, \widehat \PP_Z)$ to connect $\PP_Z$ and $\widehat \PP_Z$, then 
    \[
        \begin{aligned}
            (\text{Term A}) = &\left|\int W_{h}\left(\PP_Y^{z}, \QQ_Y^{z}\right) \diff \PP_Z(z) - \int W_{h}\left(\PP_Y^{z}, \widehat \QQ_Y^{z}\right) \diff \widehat \PP_Z(z)\right|\\
            = & \left|\int W_{h}\left(\PP_Y^{z}, \QQ_Y^{z}\right) -  W_{h}\left(\PP_Y^{z'}, \widehat \QQ_Y^{z'}\right) \diff \pi_{Z, Z'}(z, z')\right|\\
            \leq & \int \left|W_{h}\left(\PP_Y^{z}, \QQ_Y^{z}\right) -  W_{h}\left(\PP_Y^{z'}, \widehat \QQ_Y^{z'}\right) \right|\diff \pi_{Z, Z'}(z, z')\\
            \leq & L_h \int W_{1}\left(\PP_Y^{z}, \PP_Y^{z'}\right) + W_{1}\left(\QQ_Y^{z}, \widehat \QQ_Y^{z'}\right) \diff \pi_{Z, Z'}(z, z')\\
            \leq & L_h \int L_p \|z - z'\|_2  + W_{1}\left(\QQ_Y^{z}, \widehat \QQ_Y^{z'}\right) \diff \pi_{Z, Z'}(z, z')\\
            \leq & L_h \int 2 L_p \|z - z'\|_2  + W_{1}\left(\QQ_Y^{z'}, \widehat \QQ_Y^{z'}\right) \diff \pi_{Z, Z'}(z, z')\\
            \leq & 2 L_h L_p W_1(\PP_Z, \widehat \PP_Z) + L_h \int W_{1}\left(\QQ_Y^{z'}, \widehat \QQ_Y^{z'}\right) \diff \widehat \PP_{Z}(z').
        \end{aligned}
    \]
    Here, the second inequality is due to \cite[Lemma 5]{lin2025estimation}; the third inequality is due to the condition; the last inequality is due to the definition of $\pi$, which achieves the \textit{optimality} in $W_{1}(\PP_{Z}, \widehat \PP_Z)$.

    For \textbf{(Term B)}: we have
    \[
        \begin{aligned}
            (\text{Term B}) = & \left|\int W_{h}\left(\PP_Y^{z}, \widehat \QQ_Y^{z}\right) \diff \widehat \PP_Z(z) - \int W_{h}\left(\widehat \PP_Y^{z}, \widehat \QQ_Y^{z}\right) \diff \widehat \PP_Z(z) \right|\\
            \leq & \int \left| W_{h}\left(\PP_Y^{z}, \widehat \QQ_Y^{z}\right)  - W_{h}\left(\widehat \PP_Y^{z}, \widehat \QQ_Y^{z}\right) \right| \diff \widehat \PP_Z(z)\\
            \leq & L_h \int W_1(\PP_Y^{z}, \widehat \PP_Y^{z}) \diff \widehat \PP_Z(z),
        \end{aligned}
    \]
    where the last inequality is due to \cite[Lemma 5]{lin2025estimation}.
    
    Jointly, we get
    \[
        |\Cwass_h(\widehat \PP, \widehat \QQ) - \Cwass_h(\PP, \QQ)| \leq 2L_h L_p W_1(\PP_Z, \widehat \PP_Z) + L_h \int W_1(\PP_Y^z, \widehat \PP_Y^z) + W_1(\QQ_Y^z, \widehat \QQ_Y^z) \, \widehat \PP_Z(\diff z).
    \]
    
\end{proof}

\subsection{Proof of Theorem~\ref{thm:main_thm1}}
\begin{proof}[Proof of Theorem~\ref{thm:main_thm1}]
    We first verify the conditions of Proposition~\ref{prop:stability_I} under \Cref{a:boundedness_assp} to \ref{a:Lip_obj}.
    \begin{enumerate}[label=(\roman*)]
        \item Lipschitz continuity of $h$ is assumed by \Cref{a:Lip_obj}.
        \item Under \Cref{a:boundedness_assp}-\ref{a:density_smooth}, by duality, we have
        \begin{align*}
            W_1(Q_Y^z, Q_Y^{z'}) &\leq \int_{[0,1]^{d_Y}} |q(y | z) - q(y | z')| \diff y \\
            &\leq \gamma \int_{[0,1]^{d_Y}} |q(y, z) - q(y, z')| \diff y \\
            &\leq \gamma \max_{(y,z) \in [0,1]^{d_Y + d_Z}}\|\partial_z q(y,z)\|_2 \|z - z'\|_2. 
        \end{align*}
        Therefore, we set $L_p = \gamma \max_{(y,z) \in [0,1]^{d_Y + d_Z}}\|\partial_z q(y,z)\|_2 \vee \|\partial_z p(y,z)\|_2$.
            
    \end{enumerate}

    Finally, apply Proposition~\ref{prop:stability_I}, we have
    \[
        \EE[|V_{\cc} - \widehat V_{\cc, n}|] \leq 2L_h L_p \EE[W_1(\PP_Z, \widehat R_{n})] + L_h \EE\left[\int W_1(\PP_Y^z, \widehat \PP_{n,Y}^z) + W_1(\QQ_Y^z, \widehat \QQ_{n,Y}^z) \, \widehat R_{n}(\diff z)\right]. 
    \]

    For the first term, applying Jensen's inequality, we get
    \begin{align*}
        &\EE[W_1(\PP_Z, \widehat R_{n})]\\
        \leq & \frac{1}{2}(\EE[W_1(\PP_Z, \widehat P_{n, Z})] + \EE[W_1(\PP_Z, \widehat Q_{n, Z})])\\
        \leq & \frac{1}{2} \left((\EE[W_2^2(\PP_Z, \widehat P_{n, Z})])^{\frac{1}{2}} + (\EE[W_2^2(\PP_Z, \widehat Q_{n, Z})])^{\frac{1}{2}}\right).
    \end{align*}
        
    Apply a similar reasoning as bounding the (Term B) in the proof of Proposition~\ref{prop:core_rate}: 
    Under \Cref{a:boundedness_assp}, applying \Cref{prop:measure.and.density} - \ref{lem:wavelet_oper}, we get
    \begin{align*}
        \EE[W_2^2(\PP_Z, \widehat P_{n, Z})] 
        \leq C_{\textup{wb}}^2 C_{\textup{emb}}^2 \gamma \EE[\|p - \hat p\|_2^2]. 
    \end{align*}
    Then, applying \Cref{prop:L2-bound}, we get
    \[
        \EE[W_2^2(\PP_Z, \widehat P_{n, Z})] \lesssim n^{-\frac{2s}{2s+ d_Y + d_Z}}.
    \]

    For the second term, applying Jensen's inequality and Proposition~\ref{prop:core_rate}, we get
    \[
        \EE\left[\int W_1(\QQ_Y^z, \widehat \QQ_{n,Y}^z) \, \widehat R_{n}(\diff z)\right] \leq \left(\EE\left[\int W_2^2(\QQ_Y^z, \widehat \QQ_{n,Y}^z) \, \widehat R_{n}(\diff z)\right]\right)^{\frac{1}{2}} \lesssim n^{-\frac{s}{2s+d_Y + d_Z}}.
    \]

    As a result, we get
    \[
        \EE[|V_{\cc} - \widehat V_{\cc, n}|] \leq C L_h n^{-\frac{s}{2s+d_Y + d_Z}},      
    \]
    where $C$ is a constant that depends on $P, Q, \gamma, s, d_Y, d_Z$.
    
\end{proof}

\subsection{Proof of \Cref{prop:stability.COT}}
The stability bound under quadratic objective stems from the case for unconditional Wasserstein-$2$ distance.
\begin{proposition}[Proposition 12 in \parencite{manole2024plugin}]\label{prop:stability.OT}
Let $\PP, \QQ \in \mathcal{P}_{\mathrm{ac}}([0,1]^d)$. Assume that the Brenier potential $\varphi_0$ between $\PP, \QQ$ is a convex function such that
$\varphi_0 \in \calC^2([0,1]^d)$ and $\varphi_0$ is $\lambda$-strongly convex, i.e., there exists $\lambda > 0$ satisfying
\[
\frac{1}{\lambda} I_d \preceq \nabla^2 \varphi_0(x) \preceq \lambda I_d
\quad \text{for all } x \in [0,1]^d.
\]
For any $\widehat \PP, \widehat{\QQ} \in \mathcal{P}([0,1]^d)$,
\begin{equation}
0 \le W_2^2(\widehat \PP,\widehat{\QQ})
-
W_2^2(\PP,\QQ)
-
\int \phi_0 \, d(\widehat{\PP}-\PP)
-
\int \psi_0 \, d(\widehat{\QQ}-\QQ)
\;\le\;
\lambda\,(W_2(\widehat{\PP},\PP) + W_2(\widehat{\QQ},\QQ))^2,
\end{equation}
where $\phi_0(x) = \|x\|^2 - \varphi_0(x)$ and $\psi_0(x) = \|x\|^2 - \varphi_0^*(x)$, where $\varphi_0^*$ is the convex conjugate of $\varphi_0$.
\end{proposition}

\begin{remark}
We compare the conditional and the unconditional stability result. The upper bound in \Cref{prop:stability.COT} $
\int W_2^2(\widehat \PP_Y^{z}, \PP_Y^z)\,\widehat R(\mathrm dz)
+
\int W_2^2(\widehat \QQ_Y^{z}, \QQ_Y^z)\,\widehat R(\mathrm dz)$ 
is no smaller than 
$W_2^2\!\big(\widehat{\PP}_Y^z\times \widehat R,\;
\PP_Y^z\times \widehat R\big)
+
W_2^2\!\big(\widehat{\QQ}_Y^z\times \widehat R,\;
\QQ_Y^z\times \widehat R\big)$ by \Cref{lemm:CW2.V.S.W2}, where the latter is similar to the upper bound
$W_2^2(\widehat{\PP},\PP) + W_2^2(\widehat{\QQ},\QQ)$ in \Cref{prop:stability.OT}.
A larger upper bound in \Cref{prop:stability.COT} is consistent with the fact that COT is intrinsically more challenging than the unconditional counterpart.
\end{remark}

\begin{lemma}[Conditional Wasserstein distance V.S. Wasserstein distance]\label{lemm:CW2.V.S.W2}
Let $\PP$ and $\QQ$ be probability measures on $\mathcal{Y}\times\mathcal{Z}$
with the same marginal distribution $\PP_Z = \QQ_Z$ on $\mathcal{Z}$. For $z\in\mathcal{Z}$,
let $\PP_Y^z$ and $\QQ_Y^z$ denote the corresponding
conditional distributions on $\mathcal{Y}$. 
Consider the quadratic cost on
$\|y-y'\|^2$. Then
\[
\mathbb{E}_{Z\sim \PP_Z}\!\left[
W_2^2\!\big(\PP_Y^ Z,\,\QQ_Y^Z\big)
\right]
\ \ge\
W_2^2(\PP,\QQ).
\]
\end{lemma}

To obtain the strong convexity of the conditional potential functions in our setting, we shall leverage the following result.
\begin{lemma}[{\cite[Corollary 3.2]{gigli2011holder}}]\label{lem:convex_lam}
    Assume $P,Q \in \calP([0,1]^d)$ with densities $p,q$. Assume the Brenier's potential $\varphi_0$ between $P,Q$ satisfies $\varphi_0 \in C^2([0,1]^d)$ and $\gamma^{-1} \le p(x),q(x) \le \gamma\, \forall x\in [0,1]^d$ for a constant $\gamma > 0$. 
    Then there exists a constant $\lambda > 0$, depending only on $\gamma$, such that $\varphi_0$ is $\lambda$-strongly convex.
\end{lemma}

Now we are ready to prove the proposition.
\begin{proof}[Proof of \Cref{prop:stability.COT}]\label{proof:prop:stability.COT}
In the following, we denote $\phi_z(y):= \phi(y,z) $, $\psi_z(y):= \psi(y,z) $. 

First, applying \Cref{lem:convex_lam} under \Cref{a:boundedness_assp} ($\gamma$) and also by the assumption that $\|\phi(\cdot, z)\|_{\calC^2([0,1]^{d_Y})}$ is bounded by $C_2$ uniformly for $z$, we get, there exists a constant $\lambda > 0$ that depends on $\gamma$ and $C_2$, such that for any $(y, z) \in [0,1]^{d_Y + d_Z}$,
\[
    \lambda^{-1}I_{d_Y} \preceq \nabla^2_y \varphi(y,z) \preceq  \lambda I_{d_Y}.
\]

For all $z\in [0,1]^{d_Z}$, under \Cref{a:boundedness_assp} and   with constant $\lambda > 0$, apply \Cref{prop:stability.OT} to $(\PP_Y^z, \QQ_Y^z)$, we get
\begin{align*}    
    0 \leq &\underbrace{W_2(\widehat \PP_Y^{z},\widehat \QQ_Y^{z})^2 - W_2(\PP_Y^z,\QQ_Y^z)^2}_{:=\text{Term (a)}}- \underbrace{\int \phi_z \diff (\widehat \PP_Y^{z} - \PP_Y^z) - \int \psi_z \diff (\widehat \QQ_Y^{z} - \QQ_Y^z)}_{:=\text{Term (b)}} \\
    \leq& \lambda \underbrace{\left(W_2(\widehat \PP_Y^{z}, P_Y^z) + W_2(\widehat \QQ_Y^{z}, \QQ_Y^z)\right)^2}_{:=\text{Term (c)}},
\end{align*}
for any $z \in [0,1]^{d_Z}$.
To derive the final result, we integrate the term (a) to (c) in the inequality above with respect to the empirical marginal distribution $\widehat R(\diff z) = \widehat{\PP}_Z(\diff z) = \widehat{\QQ}_Z(\diff z)$. 
Here we use the alignment property that $\widehat \PP$ and $\widehat \QQ$ share the same
marginal distribution of $Z$.

\textbf{Term (a)}. By the alignment property of the estimated densities $\widehat R(\diff z) = \widehat{\PP}_Z(\diff z) = \widehat{\QQ}_Z(\diff z)$, we have 
    \[
        \int W_2(\widehat \PP_Y^{z},\widehat \QQ_Y^{z})^2  \widehat R(\diff z) = \Cwass_{2}(\widehat \PP, \widehat \QQ)^2.
    \]
    In addition, by the alignment property of the true densities $R(\diff z) = {\PP}_Z(\diff z) = {\QQ}_Z(\diff z)$, 
    \begin{align}\label{eq:proof:stability.COT.1}
    \begin{split}
       \int W_2(\PP_Y^z,\QQ_Y^z)^2  \widehat R(\diff z)  &= \int W_2(\PP_Y^z,\QQ_Y^z)^2  R(\diff z) + \int W_2(\PP_Y^z,\QQ_Y^z)^2  (\widehat R(\diff z) - R(\diff z))\\
        &= \Cwass_2(\PP, \QQ)^2 + \int W_2(\PP_Y^z,\QQ_Y^z)^2  (\widehat R(\diff z) - R(\diff z)).
    \end{split}
    \end{align}

\textbf{Term (b)}. We have
    \begin{align*}
        &\int \int \phi_z(y) (\widehat \PP_Y^z(\diff y) - P_Y^z(\diff y)) \widehat R(\diff z) \\
        =& \int \int \phi_z(y) (\widehat \PP_Y^z(\diff y) \widehat R(\diff z) - P_Y^z(\diff y) \widehat R(\diff z))\\
        =& \int \int \phi_z(y) (\widehat \PP_Y^z(\diff y) \widehat R(\diff z) - P_Y^z(\diff y) R(\diff z)) - \int \int \phi_z(y) P_Y^z(\diff y) (\widehat R(\diff z) - R(\diff z))\\
        =& \int \phi \diff (\widehat \PP - \PP) - \int \int \phi_z(y) P_Y^z(\diff y) (\widehat R(\diff z) - R(\diff z),
    \end{align*}
    where we use the definition of $\phi$ and $\psi$.
    Similarly,
    \begin{align*}
        &\int \int \psi_z(y) (\widehat \QQ_Y^z(\diff y) - \QQ_Y^z(\diff y)) \widehat R(\diff z) \\
        =& \int \psi \diff (\widehat \QQ - \QQ) - \int \int \psi_z(y) \QQ_Y^z(\diff y) (\widehat R(\diff z) - R(\diff z)).
    \end{align*}
    By the optimality relation of COT,
    \begin{align*}        
        W_2(\PP_Y^{z},\QQ_Y^z)^2 = \int \phi_z(y) \PP_Y^z(\diff y) + \int \psi_z(y) \QQ_Y^z(\diff y).
    \end{align*}
    Therefore, we have
    \begin{align}\label{eq:proof:stability.COT.2}
    \begin{split}
        &\int \int \phi_z(y) (\widehat \PP_Y^z(\diff y) - \PP_Y^z(\diff y)) \widehat R(\diff z)  + \int \int \psi_z(y) (\widehat \QQ_Y^z(\diff y) - \QQ_Y^z(\diff y)) \widehat R(\diff z) \\
        =& \int \phi \diff (\widehat \PP - \PP)  + \int \psi \diff (\widehat \QQ - \QQ) - \int W_2(\PP_Y^z, \QQ_Y^z)^2  (\widehat R(\diff z) - R(\diff z)).
        \end{split}
    \end{align}

\textbf{Term (c)}. The integration of Term (c) over $\widehat{R}(dz)$ together with the inequality $(a+b)^2 < 2a^2+2b^2$ give the right hand side of the inequality of \Cref{prop:stability.COT}.

Note that the term $\int \psi \diff (\widehat \QQ - \QQ) - \int W_2(\PP_Y^z, \QQ_Y^z)^2  (\widehat R(\diff z) - R(\diff z))$ shows up in both \eqref{eq:proof:stability.COT.1}, \eqref{eq:proof:stability.COT.2} and are thus canceled.
Combining term (a), term (b), and term (c), we get the desired result.
\end{proof}

\subsection{Proof of \Cref{theo:CLT}}\label{appe:sec:proof:theo:CLT}
First we need a lemma to prove the smoothness of $\phi, \psi$.
\begin{lemma}[Smooth Kantorovich potential]\label{lem:smooth_kan}
    Under \Cref{a:boundedness_assp}-\ref{a:density_smooth}, when $2s > d_Y + d_Z$, $\varphi \in \calC^{s+1}([0,1]^{d_Y+d_Z})$, then $\phi, \psi \in \calC^{s+1}([0,1]^{d_Y+d_Z})$.
\end{lemma}

\begin{proof}[Proof of \Cref{theo:CLT}]
    Since $s > d_Y + d_Z$, thus $s > 2$. Since $\varphi \in \calC^{s+1}([0,1]^{d_Y + d_Z})$, we apply \Cref{prop:stability.COT} and get
    \begin{align*}
        0 \leq & \Cwass_2(\widehat P_n^{\dagger},\widehat Q_n^{\dagger})^2 - \Cwass_2(\PP,\QQ)^2 - \int \phi \diff (\widehat P_n^{\dagger} - \PP) - \int \psi \diff (\widehat Q_n^{\dagger} - \QQ) \\
        \leq & 2\lambda \int W^2_2(\widehat \PP_Y^z, \PP_Y^z) + W^2_2(\widehat \QQ_Y^z, \QQ_Y^z) \widehat R_{n}(\diff z),
    \end{align*}
    where $\lambda$ is a constant that depends on $\gamma, \|\varphi\|_{\calC^2}$. 

    By \Cref{prop:core_rate}, we have 
    \[
         \EE\left[\int W^2_2(\widehat \PP_Y^z, \PP_Y^z) \, \widehat R(\diff z)\right] \vee \EE\left[\int W^2_2(\widehat \QQ_Y^z, \QQ_Y^z) \, \widehat R(\diff z)\right] \lesssim n^{-\frac{2s}{2s+d_Y+d_Z}}.
    \]
    Therefore, when $2s > d_Y+d_Z$, 
    \[
        \int W^2_2(\widehat \PP_Y^z, \PP_Y^z) + W^2_2(\widehat \QQ_Y^z, \QQ_Y^z) \widehat R_{n}(\diff z) = o_p(n^{-1/2}).
    \]
        
    It remains to show that 
    \[
        \sqrt{n} \left(\int \phi \diff (\widehat P_n^{\dagger} - \PP) +  \int \psi \diff (\widehat Q_n^{\dagger} - \QQ) \right) \overset{\textup{d}}{\rightarrow} \calN(0, \sigma^2).
    \]

We use $\hat r(z)$ to denote the density of $\widehat R_{n}(z)$.
We use $\mathcal Y \times \mathcal Z$ and $[0,1]^{d_Y+d_Z}$ interchangeably.
Since $\mathcal Y \times \mathcal Z$ is compact and $p(y,z)$, $q(y,z)$, $\phi(y,z)$, $\psi(y,z) \in \mathcal C^{s+1}(\mathcal Y \times \mathcal Z)$, there exists $M > 0$ such that $p(y,z)$, $q(y,z)$, $|\phi(y,z)|$, $|\psi(y,z)| < M$ for any $(y,z) \in \mathcal Y \times \mathcal Z$.

According to \Cref{defi:wavelet_estimator}, we can decompose $\int \phi \diff (\widehat P_n^{\dagger} - \PP)$ as
\begin{align*}
    \int \phi \diff (\widehat P_n^{\dagger} - \PP) 
    =& \underbrace{\int \phi \diff (\widehat P_n^{\dagger} -  \widehat \PP_n)}_{:=\text{Term (a)}} + \underbrace{\int \phi \diff (\widehat \PP_n - \PP)}_{:=\text{Term (b)}}.
\end{align*}

\textbf{Term (a)}. In the following, we ignore the dependence on $n$.
\begin{align*}
\int \phi \diff (\widehat P_n^{\dagger} -  \widehat \PP_n)
=& \int \phi(y,z) \left(\frac{\hat p(y,z)}{\hat p_Z(z)} \hat r(z) - \frac{\hat p(y,z)}{\hat p_Z(z)} \hat p_Z(z)  \right) \diff y \diff z  \quad (\text{\Cref{defi:wavelet_estimator}})\\
=& \int \phi(y,z) \frac{\hat p(y,z)}{\hat p_Z(z)} \left(\hat r(z) -\hat p_Z(z)\right) \diff y\diff z \\
=& \underbrace{\int \phi(y,z) \frac{p(y,z)}{p_Z(z)} \left(\hat r(z) -\hat p_Z(z)\right) \diff y\diff z}_{:=\text{Term (a.1)}} \\
&+ \underbrace{\int \phi(y,z) \left(\frac{\hat p(y,z)}{\hat p_Z(z)} - \frac{p(y,z)}{p_Z(z)} \right) \left(\hat r(z) -\hat p_Z(z)\right) \diff y\diff z}_{:=\text{Term (a.2)}}.
\end{align*}
Term (a.1) is relatively straightforward as  $\hat r(z)$ and $\hat p_Z(z)$ are based on marginalizing standard wavelet-based density estimators, and we postpone the analysis.
For Term (a.2), by the Cauchy-Schwarz inequality,
\begin{align*}
    &\left(\int \phi(y,z) \left(\frac{\hat p(y,z)}{\hat p_Z(z)} - \frac{p(y,z)}{p_Z(z)} \right) \left(\hat r(z) -\hat p_Z(z)\right) \diff y\diff z \right)^2\\
    \le &M^2 \underbrace{\int \left(\frac{\hat p(y,z)}{\hat p_Z(z)} - \frac{p(y,z)}{p_Z(z)} \right)^2 \diff y\diff z}_{:=\text{Term (a.2.1)}}  \underbrace{\int \left(\hat r(z) -\hat p_Z(z)\right)^2 \diff z}_{:=\text{Term (a.2.2)}}.
\end{align*}

For Term (a.2.1), we decompose it as
\begin{align}\label{proof:eq:a.2.1}
\begin{split}
    \text{Term (a.2.1)} 
    =& \int  \left(\frac{\hat p(y,z)}{\hat p_Z(z)} - \frac{p(y,z)}{\hat p_Z(z)}  +  \frac{p(y,z)}{\hat p_Z(z)} - \frac{p(y,z)}{p_Z(z)} \right)^2 \diff y\diff z\\
    \le & 2  \int \frac{1}{\hat p_Z(z)^2} \left(\hat p(y,z)-p(y,z)\right)^2 \diff y\diff z + 2\frac{M^2}{\hat p_Z(z)^2 p_Z(z)^2} \int  \left(\hat p_Z(z)-p_Z(z)\right)^2 \diff z.
    \end{split}
\end{align}
Then, we apply the same reasoning as in the proof of \Cref{prop:core_rate} that with probability at least $1 - c / n^2$ for some constant $c$, $\hat p_Z(z) \geq \gamma^{-1}/2\,\forall z$. Thus,
\[
    \text{Term (a.2.1)} \leq 2 (2\gamma)^4   \int \left(\hat p(y,z)-p(y,z)\right)^2 \diff y\diff z + 2M^2 \int  \left(\hat p_Z(z)-p_Z(z)\right)^2 \diff z.
\]

Again, by the same reasoning as in the proof of \Cref{prop:core_rate}, applying \Cref{prop:L2-bound}, we get 
\[
    \EE[\text{Term (a.2.1)}] \lesssim n^{-\frac{2s}{2s + d_Y + d_Z}}.
\]
When $2s > d_Y + d_Z$, $\text{Term (a.2.1)} = o_p(n^{-1/2})$.

For Term (a.2.2), since 
\begin{align*}
    \hat r(z) -\hat p_Z(z) &= (\hat r(z) - p_Z(z)) -(\hat p_Z(z)-p_Z(z)),\\
    &= \frac{1}{2}(\hat p_Z(z) - p_Z(z)) + \frac{1}{2}(\hat q_Z(z) - p_Z(z)) -(\hat p_Z(z)-p_Z(z)) & (\text{\Cref{defi:wavelet_estimator}}).
\end{align*}
Then, similarly, we have, when $2s > d_Y + d_Z$, $\text{Term (a.2.2)} = o_p(n^{-1/2})$.

Jointly, and we get $\text{Term (a.2)} = o_p(n^{-1/2})$, 
which implies that in Term (a.2) is ignorable.

Next, we bring Term (b), Term (a.1), as well as the counterparts for $Q$ together.
By the construction of $\hat p^\dagger(y,z)$ (\Cref{defi:wavelet_estimator}), we have $\hat r(z)
  = \left(\hat p_Z(z) + \hat q_Z(z)\right) / 2$, which implies \begin{align}\label{proof:eq:wavelet.marginal}
  -(\hat r(z) -\hat p_Z(z)) = \hat r(z) -\hat q_Z(z) = (\hat p_Z(z) -\hat q_Z(z))/2.
  \end{align}
Then
\begin{align*}
    &\int \phi \diff (\widehat P_n^{\dagger} - \PP) + \int \psi \diff (\widehat Q_n^{\dagger} - \QQ)\\
    =& \underbrace{\int \phi(y,z) \frac{p(y,z)}{p_Z(z)} \left(\hat r(z) -\hat p_Z(z)\right) \diff y\diff z}_{\text{Term (a.1)}} + \underbrace{\int \phi(y,z) (\hat p(y,z) - p(y,z)) \diff y \diff z}_{\text{Term (b)}}  \\
    &+  \int \psi(y,z) \frac{q(y,z)}{q_Z(z)} \left(\hat r(z) -\hat q_Z(z)\right) \diff y\diff z + \int \psi(y,z) (\hat q(y,z) - q(y,z)) \diff y \diff z + o_p(n^{-1/2})\\
    = &  \int \left(\phi(y,z) \frac{p(y,z)}{p_Z(z)}-\psi(y,z) \frac{q(y,z)}{q_Z(z)}\right) \left(\frac{\hat q_Z(z) -\hat p_Z(z)}{2}\right)\diff y\diff z \qquad (\text{by Eq.~\eqref{proof:eq:wavelet.marginal}})\\
    &+ \int \phi(y,z) (\hat p(y,z) - p(y,z)) \diff y \diff z + \int \psi(y,z) (\hat q(y,z) - q(y,z)) \diff y \diff z + o_p(n^{-1/2}) \\
   =& \underbrace{\int -\left(\phi(y,z) \frac{p(y,z)}{p_Z(z)}-\psi(y,z) \frac{q(y,z)}{q_Z(z)}\right) \left(\frac{\hat p_Z(z) - p_Z(z)}{2}\right) + \phi(y,z) (\hat p(y,z) - p(y,z))  \diff y\diff z}_{:=\text{Term (c)}} \\
   &+\underbrace{\int \left(\phi(y,z) \frac{p(y,z)}{p_Z(z)}-\psi(y,z) \frac{q(y,z)}{q_Z(z)}\right) \left(\frac{\hat q_Z(z) - q_Z(z)}{2}\right) + \psi(y,z) (\hat q(y,z) - q(y,z))  \diff y\diff z}_{:=\text{Term (d)}}+ o_p(n^{-1/2})\\
    &\quad \quad \quad \quad (p_Z(z)=q_Z(z)~\text{and}~\hat q_Z(z) - \hat p_Z(z)=\hat q_Z(z) - q_Z(z) - (\hat p_Z(z) - p_Z(z)))\\
\end{align*}
For Term (c), as $\hat p_Z(z) - p_Z(z) = \int \hat p(y,z) - p(y,z) \diff y \diff z$,
\begin{align*}
    \text{term (c)}
     =& \int \left(\int -\left(\phi(y',z) \frac{p(y',z)}{p_Z(z)}-\psi(y',z) \frac{q(y',z)}{q_Z(z)}\right) \diff y'\right) \left(\int \left(\frac{\widehat   p(y, z) - p(y,z)}{2}\right) \diff y \right) \diff z\\
     &+ \int \phi(y,z) (\hat p(y,z) - p(y,z))  \diff y\diff z\\
     =& \int \underbrace{\left(\phi(y,z)-\frac{1}{2}\left(\int\left(\phi(y',z) \frac{p(y',z)}{p_Z(z)}-\psi(y',z) \frac{q(y',z)}{q_Z(z)}\right) \diff y'\right) \right)}_{=\eta(y,z)} (\hat p(y,z) - p(y,z))  \diff y\diff z.
\end{align*}
As $\phi(y,z)$, $\psi(y,z)$, $p(y,z)$, $q(y,z) \in \mathcal{C}^{s+1}(\mathcal Y \times \mathcal Z)$ and $p_Z, q_Z \geq \gamma^{-1}$, we have $\eta(y,z) \in \mathcal C^{s+1}(\mathcal Y \times \mathcal Z)$.
By Lemma 11 (control bias) and Lemma 66 (central limit theorem) of \parencite{manole2024plugin}, we establish the central limit theorem for $\sqrt{n} \times$~term (c) with the variance
\begin{align*}
    2\var_{\PP}\left(\eta(y,z)\right)
    =2\var_{\PP}\left(\phi(Y,Z) - \frac{\EE_{\PP}[\phi(Y,Z) \mid Z] - \EE_{\QQ}[\psi(Y,Z) \mid Z])}{2}\right).
\end{align*} 
The same argument applies to term (d). 
Note that term (c) and term (d) are based on different observations and thus independent, which yields a central limit theorem for their sum scaled by $\sqrt{n}$ with the asymptotic variance $2 \left(\var_{\PP}\left(\eta(y,z)\right) + \var_{\QQ}\left(\kappa(y,z)\right)\right)$.
\end{proof}

\subsection{Proof of Proposition~\ref{prop:embedding}}
    
\begin{proof}[Proof of Proposition~\ref{prop:embedding}]
The first part $\calC^s([0,1]^d) \subseteq \mathcal{B}^s_{\infty,\infty}([0,1]^{d})$ has been proved in \cite[Lemma 25]{manole2024plugin}. We focus on the second part: $\calC^s \subseteq \mathcal{B}^{s - \epsilon}_{2,2}$ for any $\epsilon > 0$.

    \begin{lemma}[H\"older to Sobolev/Besov with $\varepsilon$-loss]\label{lem:Holder_to_Besov_eps}
Let $d, s>0$. For any $\varepsilon>0$, there exists a constant $C=C(d,s,\varepsilon)$ such that
\[
\|f\|_{B^{\,s-\varepsilon}_{2,2}([0,1]^d)} \le C\,\|f\|_{\calC^{s}([0,1]^d)}
\qquad \forall f\in \calC^{s}([0,1]^d).
\]
Equivalently, $\calC^{s}([0,1]^d)\hookrightarrow B^{\,s-\varepsilon}_{2,2}([0,1]^d)$.
\end{lemma}

\begin{proof}
Fix $\varepsilon>0$ and set $t=s-\varepsilon$. Let $\{\phi_k\}_{k\in\calI_0}$ be the (boundary-corrected) scaling functions at the coarsest resolution and
$\{\psi_{j,k}\}_{j\ge j_0,\;k\in\calI_j}$ be a compactly supported, boundary-corrected wavelet basis on $[0,1]^d$ with regularity strictly larger than $s$.
Then every $f\in L^2([0,1]^d)$ admits the wavelet expansion
\[
f
=
\sum_{k\in\calI_0} \alpha_k\,\phi_k
\;+\;
\sum_{j\ge j_0}\sum_{k\in\calI_j} \beta_{j,k}\,\psi_{j,k},
\qquad
\alpha_k=\langle f,\phi_k\rangle,\ \beta_{j,k}=\langle f,\psi_{j,k}\rangle .
\]
We use the standard wavelet characterizations of Besov norms on $[0,1]^d$ (with boundary correction):
\begin{align}
\|f\|_{B^{t}_{2,2}([0,1]^d)}^2
&\lesssim
\sum_{k\in\calI_0}|\alpha_k|^2
+
\sum_{j\ge j_0} 2^{2jt}\sum_{k\in\calI_j}|\beta_{j,k}|^2 .
\label{eq:Besov_wavelet_char}
\end{align}
Let $M:=\|f\|_{\calC^{s}([0,1]^d)}$. By the proof of \cite[Lemma 25]{manole2024plugin}, we have
\begin{equation}\label{eq:beta_decay}
\sup_{k\in\calI_j}|\beta_{j,k}|
\le
C_1\, M\, 2^{-j(s+d/2)}
\qquad \forall j\ge j_0,
\end{equation}
for some constant $C_1$ depending only on $(d,s)$ and the chosen wavelet system.

Next, recall that the number of wavelets at level $j$ satisfies $|\calI_j|\lesssim 2^{jd}$.
Using \eqref{eq:beta_decay},
\begin{equation}\label{eq:l2_bound_each_scale}
\sum_{k\in\calI_j}|\beta_{j,k}|^2
\le
|\calI_j|\Big(\sup_{k\in\calI_j}|\beta_{j,k}|\Big)^2
\lesssim
2^{jd}\cdot \big(M^2 2^{-2j(s+d/2)}\big)
=
C_2\,M^2\,2^{-2js},
\end{equation}
for some constant $C_2$.

Multiplying \eqref{eq:l2_bound_each_scale} by $2^{2jt}=2^{2j(s-\varepsilon)}$ yields
\[
2^{2jt}\sum_{k\in\calI_j}|\beta_{j,k}|^2
\le
C_2\,M^2\,2^{2j(s-\varepsilon)}\cdot 2^{-2js}
=
C_2\,M^2\,2^{-2j\varepsilon}.
\]
Summing over $j\ge j_0$ and using $\sum_{j\ge j_0}2^{-2j\varepsilon}<\infty$, we obtain
\begin{equation}\label{eq:high_freq_sum}
\sum_{j\ge j_0} 2^{2jt}\sum_{k\in\calI_j}|\beta_{j,k}|^2
\le
C_3\,M^2,
\end{equation}
for a constant $C_3=C_3(d,s,\varepsilon)$.

Finally, again by \cite[Lemma 2]{manole2024plugin}, $\sup_{k\in\calI_0}|\alpha_k| \leq M$, thus
\begin{equation}\label{eq:low_freq_sum}
\sum_{k\in\calI_0}|\alpha_k|^2
\le
|\calI_0|\Big(\sup_{k\in\calI_0}|\alpha_k|\Big)^2
\le
C_4\,M^2.
\end{equation}
Combining \eqref{eq:Besov_wavelet_char}--\eqref{eq:low_freq_sum} gives
\[
\|f\|_{B^{t}_{2,2}([0,1]^d)}^2
\le
C_5\,M^2,
\]
hence $\|f\|_{B^{s-\varepsilon}_{2,2}([0,1]^d)} \le \sqrt{C_5}\,\|f\|_{\calC^{s}([0,1]^d)}$, proving the claim.
\end{proof}

\end{proof}

\subsection{Proof of Proposition~\ref{prop:L2-bound}}
In this section, we prove Proposition~\ref{prop:L2-bound} in a more general way (see Theorem~\ref{thm:L2_rate_trunc}). By Assumption~\ref{a:density_smooth} and Proposition~\ref{prop:embedding}, $p, q \in  \mathcal{B}^s_{2,2}([0,1]^{d_Y+d_Z})$ and $\gamma^{-1}\leq p,q \leq \gamma$, thus satisfying Assumptions~\ref{ass:p_smooth}--\ref{ass:wavelet}.
\begin{assumption}[Smoothness and boundedness]\label{ass:p_smooth}
Let $p$ be a density on $[0,1]^d$ such that $p \in B^{s}_{2,2}([0,1]^d)$ for some $s>0$.
Assume also $0 \le p \le M$ a.e.\ for some $M<\infty$.
\end{assumption}

\begin{assumption}[{Wavelet basis on $[0,1]^d$}]\label{ass:wavelet}
Let $\{\phi_{j_0,k}^{\mathrm{bc}}\}_{k\in\calK_0(j_0)} \cup \{\psi^{\mathrm{bc}}_{j,k,\ell}\}_{j\ge j_0,\,k\in\calK_0(j),\,\ell\in\{0,1\}^d\setminus\{0\}}$
be an orthonormal (boundary-corrected) wavelet basis of $L^2([0,1]^d)$.
Write the coefficients of $p$ as
\[
\alpha_{k} := \langle p, \phi_{j_0,k}^{\mathrm{bc}}\rangle,
\qquad
\beta_{j,k,\ell} := \langle p, \psi^{\mathrm{bc}}_{j,k,\ell}\rangle.
\]
\end{assumption}

\paragraph{(Linear projection estimator).}
Let $X_1,\dots,X_n \sim p$ i.i.d.
Fix an integer $J\ge j_0$ and define the empirical coefficients
\[
\hat\alpha_{k} := \frac1n \sum_{i=1}^n \phi_{j_0,k}^{\mathrm{bc}}(X_i),
\qquad
\hat\beta_{j,k,\ell} := \frac1n \sum_{i=1}^n \psi^{\mathrm{bc}}_{j,k,\ell}(X_i),\quad j_0\le j\le J.
\]
Define the (untruncated) projection estimator
\[
\tilde p_J(x)
:=
\sum_{k\in\calK(j_0)} \hat\alpha_k\,\phi_{j_0,k}^{\mathrm{bc}}(x)
\;+\;
\sum_{j=j_0}^{J}\ \sum_{\ell\neq 0}\ \sum_{k\in\calK(j)}
\hat\beta_{j,k,\ell}\,\psi^{\mathrm{bc}}_{j,k,\ell}(x).
\]
Define the truncated (clipped) estimator
\[
\tilde p_J^{\dagger}(x) := \big(\tilde p_J(x)\big)_{[0,M]}
:= \min\{M,\max\{0,\tilde p_J(x)\}\}.
\]

\begin{theorem}[$L^2$ risk of the truncated projection estimator]\label{thm:L2_rate_trunc}
Under Assumptions~\ref{ass:p_smooth}--\ref{ass:wavelet},
\[
\EE\|\tilde p_J^{\dagger} - p\|_{L^2([0,1]^d)}^2
\;\lesssim\;
2^{-2Js} \;+\; \frac{2^{Jd}}{n}.
\]
Consequently, choosing $2^J \asymp n^{1/(2s+d)}$ yields
\[
\EE\|\tilde p_J^{\dagger} - p\|_2^2 \;\lesssim\; n^{-\frac{2s}{2s+d}},
\]
and moreover the risk is always bounded by $\lesssim n^{-1}$, hence
\[
\EE\|\tilde p_J^{\dagger} - p\|_2^2 \;\lesssim\; n^{-\frac{2s}{2s+d}}.
\]
The hidden constants depend only on $(s,d)$, the wavelet family, and $M$.
\end{theorem}

\begin{proof}
\textbf{Step 0 (Truncation cannot increase $L^2$ error).}
The clipping map $T(u)=(u)_{[0,M]}$ is $1$-Lipschitz on $\RR$, i.e.\ $|T(u)-T(v)|\le |u-v|$.
Therefore,
\[
\|\tilde p_J^{\dagger} - p\|_2^2
=
\|T(\tilde p_J)-T(p)\|_2^2
\le
\|\tilde p_J - p\|_2^2,
\]
since $p\in[0,M]$ a.e.\ by assumption. Hence it suffices to bound $\EE\|\tilde p_J-p\|_2^2$.

\textbf{Step 1 (Orthogonal bias--variance decomposition).}
By orthonormality,
\[
\|\tilde p_J - p\|_2^2
=
\|\tilde p_J - p_J\|_2^2 + \|p_J - p\|_2^2,
\]
and taking expectations gives
\[
\EE\|\tilde p_J - p\|_2^2
=
\EE\|\tilde p_J - p_J\|_2^2 + \|p_J - p\|_2^2.
\]

\textbf{Step 2 (Bias bound).}
Since $p\in B^s_{2,2}([0,1]^d)$,
the wavelet characterization yields the tail energy bound
\[
\|p-p_J\|_2^2
=
\sum_{j>J}\ \sum_{\ell\neq 0}\ \sum_{k\in\calK(j)} \beta_{j,k,\ell}^2
\;\lesssim\;
2^{-2Js}\,\|p\|_{B^s_{2,2}}^2.
\]

\textbf{Step 3 (Variance bound).}
By orthonormality,
\[
\EE\|\tilde p_J - p_J\|_2^2
=
\sum_{k\in\calK(j_0)} \EE(\hat\alpha_k-\alpha_k)^2
\;+\;
\sum_{j=j_0}^{J}\ \sum_{\ell\neq 0}\ \sum_{k\in\calK(j)}
\EE(\hat\beta_{j,k,\ell}-\beta_{j,k,\ell})^2.
\]
Each empirical coefficient is an average of i.i.d.\ terms, so
\[
\EE(\hat\beta_{j,k,\ell}-\beta_{j,k,\ell})^2
=
\frac{1}{n}\textup{Var}\!\big(\psi^{\mathrm{bc}}_{j,k,\ell}(X_1)\big)
\le
\frac{1}{n}\EE\big[\psi^{\mathrm{bc}}_{j,k,\ell}(X_1)^2\big]
\le
\frac{\|p\|_{\infty}}{n}\int_{[0,1]^d} \big(\psi^{\mathrm{bc}}_{j,k,\ell}(x)\big)^2\,dx
=
\frac{\|p\|_{\infty}}{n},
\]
using $\|p\|_\infty\le M$ and $\|\psi^{\mathrm{bc}}_{j,k,\ell}\|_2=1$.
The same bound holds for the scaling coefficients.
Hence
\[
\EE\|\tilde p_J - p_J\|_2^2
\;\lesssim\;
\frac{1}{n}\Big(|\calK(j_0)| + \sum_{j=j_0}^{J}\sum_{\ell\neq 0}|\calK_0(j)|\Big).
\]
On $[0,1]^d$, $|\calK(j)|\asymp 2^{jd}$ and there are $(2^d-1)$ wavelet types $\ell\neq 0$, so
\[
\EE\|\tilde p_J - p_J\|_2^2 \;\lesssim\; \frac{1}{n}\sum_{j=j_0}^{J} 2^{jd}
\;\lesssim\; \frac{2^{Jd}}{n}.
\]

\textbf{Step 4 (Combine and optimize).}
Combining Steps 2--3 gives
\[
\EE\|\tilde p_J - p\|_2^2 \;\lesssim\; 2^{-2Js} + \frac{2^{Jd}}{n}.
\]
Since $\EE\|\tilde p_J^{\dagger}-p\|_2^2\le \EE\|\tilde p_J-p\|_2^2$, the same bound holds for $\tilde p_J^{\dagger}$.
Choosing $2^J\asymp n^{1/(2s+d)}$ balances the two terms and yields
$\EE\|\tilde p_J^{\dagger}-p\|_2^2 \lesssim n^{-2s/(2s+d)}$.
\end{proof}

\subsection{Proof of Lemmas}
\subsubsection{Proof of \Cref{prop:measure.and.density}}
\begin{proof}[Proof of \Cref{prop:measure.and.density}]
    Plug $p=1/2$ in Theorem 4 in \parencite{niles2022minimax} and we finish the proof.
\end{proof}
\subsubsection{Proof of \Cref{lem:wavelet_oper}}
\begin{proof}[Proof of \Cref{lem:wavelet_oper}]
Let $\{\psi_{\lambda}\}_{\lambda\in\Lambda}$ be an orthonormal wavelet basis of $L^2([0,1]^{d})$, and let $\Lambda=\bigsqcup_{j\ge j_0}\Lambda_j$ denote the decomposition of indices by scale. Fix $f\in L^2([0,1]^{d})$, write its wavelet expansion
\[
f \;=\; f_{<j_0} \;+\; \sum_{j\ge j_0} f_j,
\qquad
f_j \;=\; \sum_{\lambda\in\Lambda_j}\alpha(\lambda)\,\psi_{\lambda},
\]
where $f_{<j_0}$ is the coarse (scaling) component.

For each $j\ge j_0$, since $\{\psi_{\lambda}\}_{\lambda\in\Lambda_j}$ is an orthonormal basis at level $j$, Parseval's identity then gives
\begin{equation}\label{eq:parseval_Wj}
\|f_j\|_{L^2([0,1]^{d})}^2
\;=\;
\sum_{\lambda\in\Lambda_j}|\alpha(\lambda)|^2,
\qquad\text{hence}\qquad
\|f_j\|_{L^2([0,1]^{d})}
\;=\;
\Big(\sum_{\lambda\in\Lambda_j}|\alpha(\lambda)|^2\Big)^{1/2}.
\end{equation}
Therefore the left-hand side can be rewritten as
\[
\Bigg(\sum_{j\ge j_0} 2^{-j}\Big(\sum_{\lambda\in\Lambda_j}|\alpha(\lambda)|^2\Big)^{1/2}\Bigg)^2
\;=\;
\Bigg(\sum_{j\ge j_0}2^{-j}\,\|f_j\|_{L^2([0,1]^{d})}\Bigg)^2.
\]
Apply the Cauchy--Schwarz inequality to the sequences
$a_j = 2^{-j}$ and $b_j=\|f_j\|_{L^2([0,1]^{d})}$:
\begin{equation}\label{eq:cs_step}
\sum_{j\ge j_0}2^{-j}\,\|f_j\|_{L^2([0,1]^{d})}
\;\le\;
\Bigg(\sum_{j\ge j_0}2^{-2j}\Bigg)^{1/2}
\Bigg(\sum_{j\ge j_0}\|f_j\|_{L^2([0,1]^{d})}^2\Bigg)^{1/2}.
\end{equation}
Squaring \eqref{eq:cs_step} yields
\begin{equation}\label{eq:cs_step_sq}
\Bigg(\sum_{j\ge j_0}2^{-j}\,\|f_j\|_{L^2([0,1]^{d})}\Bigg)^2
\;\le\;
\Bigg(\sum_{j\ge j_0}2^{-2j}\Bigg)
\Bigg(\sum_{j\ge j_0}\|f_j\|_{L^2([0,1]^{d})}^2\Bigg).
\end{equation}

Next, by the orthogonal decomposition
\[
\|f\|_{L^2([0,1]^{d})}
=
(\|f_{<j_0}\|_{L^2([0,1]^{d})}^2+\sum_{j\ge j_0}\|f_j\|_{L^2([0,1]^{d})}^2)^{\frac{1}{2}}
\;\ge\;
\frac{1}{2} \left(\|f_{<j_0}\|_{L^2([0,1]^{d})} + \left(\sum_{j\ge j_0}\|f_j\|^2_{L^2([0,1]^{d})}\right)^{\frac{1}{2}}\right).
\]
Substituting this bound into \eqref{eq:cs_step_sq} gives
\[
\Bigg(\sum_{j\ge j_0}2^{-2j}\Bigg)^{\frac{1}{2}} \|f_{<j_0}\|_{L^2([0,1]^{d})} + \sum_{j\ge j_0}2^{-j}\,\|f_j\|_{L^2([0,1]^{d})} 
\;\le\;
2\Bigg(\sum_{j\ge j_0}2^{-2j}\Bigg)^{\frac{1}{2}}\,\|f\|_{L^2([0,1]^{d})}.
\]
Finally, combining this with
\[
    \|f\|_{\mathcal B^{-1}_{2,1}([0,1]^d)} = \|f_{<j_0}\|_{L^2([0,1]^{d})} + \sum_{j\ge j_0}2^{-j}\,\|f_j\|_{L^2([0,1]^{d})}
\]
concludes the proof.
\end{proof}

\subsubsection{Proof of Lemma~\ref{lemm:CW2.V.S.W2}}

\begin{proof}[Proof of \Cref{lemm:CW2.V.S.W2}]
For each $z\in\mathcal Z$, let $\pi_Z^*$ be an optimal coupling between
$\PP_Y^z$ and $\QQ_Y^z$ for the cost $\|y-y'\|^2$.
Define a coupling $\pi^*$ on $(\mathcal Y\times\mathcal Z)^2$ via
\[
\pi^*(dy,dz,dy',dz')
:=
\PP_Z(dz)\,\pi_Z^*(dy,dy')\,\delta_Z(dz').
\]
By construction, the first marginal of $\pi^*$ is $\PP$ and the
second marginal is $\QQ$, hence $\pi^*\in\Pi(\PP,\QQ)$.
Under the cost $c\big((y,z),(y',z')\big)=\|y-y'\|^2$, the transport cost of $\pi$ is
\[
\int \|y-y'\|^2\, d\pi^*
=
\int \left(\int \|y-y'\|^2\, d\pi_Z^*(y,y')\right)\PP_Z(dz)
=
\EE_{Z\sim \PP_Z}\!\left[
W_2^2\!\big(\PP(Y\mid Z),\,\QQ(Y\mid Z)\big)
\right].
\]
Since $W_2^2(\PP,\QQ)$ is the infimum of $\int \|y-y'\|^2\,d\pi$ over all
$\pi\in\Pi(\PP,\QQ)$, we conclude that
\[
W_2^2(\PP,\QQ)
\le
\int \|y-y'\|^2\, d\pi^*
=
\EE_{Z\sim \PP_Z}\!\left[
W_2^2\!\big(\PP(Y\mid Z),\,\QQ(Y\mid Z)\big)
\right].
\]
\end{proof}

\subsubsection{Proof of Lemma~\ref{lem:smooth_kan}}
\begin{proof}[Proof of Lemma~\ref{lem:smooth_kan}]
    Since $\phi(y, z) = \|y\|_2^2 - \varphi(y, z)$, then $\phi \in \calC^{s+1}([0,1]^{d_Y+d_Z})$. 
    
    As for $\psi$, note that $\psi(y, z) = \|y\|_2^2 - \varphi^*(y, z)$, and
    \[
        \nabla^2_y \varphi^*(y, z) = \nabla^2_y \varphi( \nabla_y \varphi^{*}(y,z), z)^{-1} \qquad (\star)
    \]
    where the convex conjugate is defined with respect to $y$, i.e., $\varphi^*(y,z) := \sup_{y' \in [0,1]^{d_Y}} \{y^\top y' - \varphi(y', z)\}$.

    By \Cref{lem:convex_lam}, there is a constant $\lambda > 0$, such that
    \[
        \frac{1}{\lambda} I_{d_Y} \preceq \nabla^2_y \varphi(y,z) \preceq \lambda I_{d_Y}\quad \forall y, z \in [0,1]^{d_Y + d_Z}.
    \]
    Therefore, by $(\star)$ and $\varphi \in \calC^{s+1}([0,1]^{d_Y+d_Z})$, we get $\varphi^* \in \calC^{s+1}([0,1]^{d_Y+d_Z})$, so does $\psi$.
\end{proof}

\section{Algorithms}\label{appe:sec:algorithms}

\subsection{Wavelet-based Density Estimation}\label{sec:optimal.transport.estimated.density}

We provide a detailed Algorithm~\ref{alg:bc_wavelet_joint_density} for the wavelet-based density estimation
described in \Cref{defi:wavelet.density.estimator}.

\begin{algorithm}[tbp]
\KwIn{Samples $(Y_i,Z_i)\in[0,1]^{d_Y+d_Z}$ for $i=1,\dots,n$; wavelet family with smoothness $s$; coarse level $J_0 \geq \lfloor\log_2(s/0.18 + 1)\rfloor + 1$.}
\KwOut{Estimated joint density $\widehat f_{Y,Z}$ on $[0,1]^{d_Y+d_Z}$.}

\BlankLine
\textbf{Notation.}
Let $d := d_Y + d_Z$ and $X_i := (Y_i,Z_i)\in[0,1]^d$.
Let $\{\phi_{J_0,k}^{\mathrm{bc}}\}_{k\in\mathcal K_{J_0}}$ be the \emph{boundary-corrected} scaling functions at level $J_0$,
and for each $j\ge J_0$ let $\{\psi_{j,k,\ell}^{\mathrm{bc}}\}_{k\in\mathcal K_j,\ \ell\in\mathcal L}$ be the boundary-corrected wavelets
(where $\ell$ indexes the $2^d-1$ multivariate wavelet types in $d$ dimensions).
Basis functions are supported on $[0,1]^d$ and form an orthonormal basis of $L^2([0,1]^d)$.

\BlankLine
\textbf{Step 1: Choose resolution levels.}\\
Pick $(J_0,J)$ such that $J_0$ fixed small, $J \geq J_0$, and $J \asymp \frac{\log_2(n)}{2s + d_Y + d_Z}$.

\BlankLine
\textbf{Step 2: Compute empirical wavelet coefficients.}\\
\ForEach{$k\in\mathcal K_{J_0}$}{
    $\widehat\alpha_{J_0,k} \leftarrow \frac{1}{n}\sum_{i=1}^n \phi_{J_0,k}^{\mathrm{bc}}(X_i)$\;
}
\For{$j=J_0$ \KwTo $J$}{
  \ForEach{$\ell\in\mathcal L$}{
    \ForEach{$k\in\mathcal K_j$}{
      $\widehat\beta_{j,k,\ell} \leftarrow \frac{1}{n}\sum_{i=1}^n \psi_{j,k,\ell}^{\mathrm{bc}}(X_i)$\;
    }
  }
}

\BlankLine
\textbf{Step 3 (optional): Threshold / shrink detail coefficients.}\\
\For{$j=J_0$ \KwTo $J$}{
  \ForEach{$\ell\in\mathcal L$}{
    \ForEach{$k\in\mathcal K_j$}{
      $\widetilde\beta_{j,k,\ell} \leftarrow \mathcal T(\widehat\beta_{j,k,\ell}; \lambda_j)$\;
    }
  }
}
\If{no thresholding}{
  $\widetilde\beta_{j,k,\ell} \leftarrow \widehat\beta_{j,k,\ell}$ for all $(j,k,\ell)$\;
}

\BlankLine
\textbf{Step 4: Assemble the estimator.}\\
Define, for any $x\in[0,1]^d$,
\[
\widehat f_{Y,Z}(x)
\;:=\;
\sum_{k\in\mathcal K_{J_0}}\widehat\alpha_{J_0,k}\,\phi_{J_0,k}^{\mathrm{bc}}(x)
\;+\;
\sum_{j=J_0}^{J}\ \sum_{\ell\in\mathcal L}\ \sum_{k\in\mathcal K_j}
\widetilde\beta_{j,k,\ell}\,\psi_{j,k,\ell}^{\mathrm{bc}}(x).
\]

\BlankLine
\textbf{Step 5: Enforce density constraints.}\\
\If{nonnegativity flag}{
  $\widehat f_{Y,Z}(x)\leftarrow \max\{\widehat f_{Y,Z}(x),0\}$\;
}

\If{normalization flag}{
  $c \leftarrow \int_{[0,1]^d}\widehat f_{Y,Z}(x)\,dx$\;
  $\widehat f_{Y,Z}(x)\leftarrow \widehat f_{Y,Z}(x)/c$\;
}

\Return{$\widehat f_{Y,Z}$}\;

\caption{Boundary-corrected wavelet density estimator for the joint density of $(Y,Z)$}
\label{alg:bc_wavelet_joint_density}
\end{algorithm}

\subsection{Optimal Transport Problem with Estimated Densities}

Given the estimated joint densities $\widehat{\QQ}$ and $\widehat{\PP}$ of $(Y_i(0), Z_i)$ and $(Y_i(1), Z_i)$
together with the estimated marginal density of $Z_i$, we generate synthetic
observations via a two-step hierarchical procedure.
\begin{enumerate}
    \item We first draw $N_1$ covariates $Z$ from
its marginal distribution.
    \item Conditional on each sampled $Z$, draw outcomes from the corresponding conditoinal distributino $\widehat{\PP}(Y(1)\mid Z)$, $\widehat{\QQ}(Y(1)\mid Z)$ to get $N_2$ obseratino for each group.
\end{enumerate}
For estimation, we first for each $Z$ solve an optimal transport problem. We then average over $Z$ to get the final estimation.

Given the estimated joint densities $\widehat{P}$ and $\widehat{Q}$
of $(Y_i(1), Z_i)$ and $(Y_i(0), Z_i)$, together with an estimate of the
marginal density of $Z$, we generate synthetic data through a hierarchical
sampling scheme designed to mimic the joint structure of the population.
\begin{enumerate}
    \item Draw $N_1$ samples of $Z$ from its estimated marginal distribution.
    \item Conditional on each sampled value of $Z$, draw $N_2$ outcomes from the conditional densities $\widehat{P}(Y(1)\mid Z)$ and $\widehat{Q}(Y(0)\mid Z)$.
\end{enumerate}
We emphasize that these
generated samples are not the original observations, but are Monte Carlo draws from the estimated density.
To estimate the optimal objective value, we solve an optimal transport problem separately at each
sampled $Z$, and average the resulting objective values over the sampled
covariates to obtain the final estimate.
Details are summarized in Algorithm~\ref{alg:example}.

There are a variety of approaches for sampling from an estimated density \parencite{robert1999monte}, including rejection sampling, importance sampling,
Markov chain Monte Carlo (MCMC) methods such as Metropolis--Hastings and
Gibbs sampling.
Algorithm~\ref{alg:example}, we adopt a discretization-based sampler for its favorable practical stability. 
Specifically, we first discretize the support of the estimated
density into a fine grid. The estimated density is then evaluated on this grid
and normalized to form a discrete probability distribution. Samples are subsequently drawn from this grid according to the discrete probabilities.
We remark that this discretization step differs from that used by \parencite{lin2025estimation}, where the discretization is used as part of the estimation procedure and applied to the observed dataset. 
In contrast, here the discretization is used as a post-estimation device for sampling from the already constructed
continuous density estimate. 
As a result, it does not affect the statistical properties of the estimator, but only serves as a practical mechanism for
Monte Carlo approximation of the Wasserstein distance.

\begin{algorithm}[tbp]
\caption{Smooth conditional Wasserstein distance}
\label{alg:example}

\KwIn{i.i.d.\,sample $((Y_i(0), Z_i(0)), 1\leq i \leq n)$ drawn from $P$, i.i.d.\,sample $((Y_j(1), Z_j(1)), 1\leq j \leq m)$ drawn from $Q$}
\KwOut{$\widehat {\Cwass}(\widehat P_n^{\dagger}, \widehat Q_n^{\dagger})$}

\BlankLine
\textbf{Step 1:}\\
Compute the wavelet estimator $\widehat P$ based on $((Y_i(0), Z_i(0)), 1\leq i \leq n)$ using Algorithm~\ref{alg:bc_wavelet_joint_density}. \\
Compute the wavelet estimator $\widehat Q$ based on $((Y_j(1), Z_j(1)), 1\leq j \leq m)$ using Algorithm~\ref{alg:bc_wavelet_joint_density}. 

\BlankLine
\textbf{Step 2:}\\
Compute the wavelet estimator of $\widehat R = (\widehat P_Z + \widehat Q_Z) / 2$.\\
Denote $\widehat P_n^{\dagger} = \widehat P_n^{\dagger}(\diff y, \diff z) = \widehat \PP_Y^z(\diff y) \widehat R(\diff z)$.\\
Denote $\widehat Q_n^{\dagger} = \widehat Q_n^{\dagger}(\diff y, \diff z) = \widehat \QQ_Y^z(\diff y)  \widehat R(\diff z)$.

\BlankLine
\textbf{Step 3:}\\
Sample $(\widetilde Z_\tau, 1 \leq \tau \leq N_Z := \lfloor(n \vee m)\log(n \vee m)\rfloor$ from  $\widehat R$.

\BlankLine
\textbf{Step 4:}\\
\For{$\tau = 1$ \KwTo $N_Z$}{
    Let $z = \tilde Z_\tau$, $N_Y = \lfloor  (n \vee m)^{(d_Y/4) \vee 1}\log(n \vee m) \rfloor$; \\
    Sample $(Y^{\tau}_k(0), 1 \leq k \leq N_Y)$ from $\widehat \PP_Y^{z}$;
    \\
    Sample $(Y^{\tau}_l(1), 1 \leq l \leq N_Y)$ from $\widehat \QQ_Y^{z}$; \\
    Compute 
    \[
        \widehat W_2^2(\widehat P_{\widetilde Z_\tau}, \widehat Q_{\widetilde Z_\tau}) := W_2^2\left(\frac{1}{N_Y} \sum_{k=1}^{N_Y} \delta_{Y^{\tau}_k(0)}, \frac{1}{N_Y} \sum_{l=1}^{N_Y} \delta_{Y^{\tau}_l(1)}\right)
    \]
    using the Python Optimal Transport (\texttt{POT}\footnote{\url{https://pythonot.github.io/}}) library \parencite{flamary2021pot};
    
}

\BlankLine
\textbf{Step 5:}\\
Define
\[
    \widehat {\Cwass}(\widehat P_n^{\dagger}, \widehat Q_n^{\dagger}) := \frac{1}{N_Z}\sum_{\tau=1}^{N_Z}\widehat W_2^2(\widehat P_{\widetilde Z_\tau}, \widehat Q_{\widetilde Z_\tau}).
\]

\Return{$\widehat {\Cwass}(\widehat P_n^{\dagger}, \widehat Q_n^{\dagger})$}
\end{algorithm}

\section{Additional Numerical Experiment}\label{appe:sec:numerical.experiment}

\subsection{Simulation Details of Estimation}

\subsubsection{Wavelet Algorithm Parameters}

\begin{itemize}
    \item \texttt{wavelet = "db4"}: Daubechies-4 wavelet basis, chosen for compact support and good smoothness properties.

    \item \texttt{mode = "periodization"}: boundary handling rule that treats the support as periodic, avoiding boundary artifacts for data supported on $[0,1]^d$.

    \item \texttt{J\_joint = 4}: wavelet resolution level for estimating the joint density $\widehat p_{Y,Z}(y,z)$. 
    Larger values lead to finer resolution (lower bias, higher variance) and constitute the primary smoothing parameter.

    \item \texttt{J\_z = 6}: wavelet resolution level for estimating the marginal density $\widehat p_Z(z)$. 
    A higher level is used since $Z$ is lower dimensional and easier to estimate accurately.
    
    \item \texttt{threshold = None}: no coefficient thresholding is applied, preserving the plug-in estimator form and simplifying theoretical analysis.

    \item \texttt{nonnegativity = True}: enforces nonnegativity of the reconstructed density, preventing oscillations from producing negative values.

    \item \texttt{renormalize = True}: rescales the estimated density to integrate to one, ensuring it is a valid probability density.

     \item \texttt{(i) $N_z = 200,\; N_y = 300$ (location model), (ii) $N_z = 50,\; N_y = 300$ (quadratic model), (iii) $N_z = 100,\; N_y = 600$ (scale model)}: number of grid points used to discretize the supports of $Z$ and $Y$ when numerically evaluating the conditional Wasserstein functional.
\end{itemize}

\subsubsection{Additional Estimation Result}
In Figure~\ref{fig:convergence_rateII}, we show the plots of convergence rate corresponding to Figure~\ref{fig:convergence_rate_loglog}.
\begin{figure*}[tbp]
        \centering
        \begin{minipage}{0.32\textwidth}
                \centering
                \includegraphics[clip, trim = 0cm 0cm 0cm 0.75cm, width = 1\textwidth]{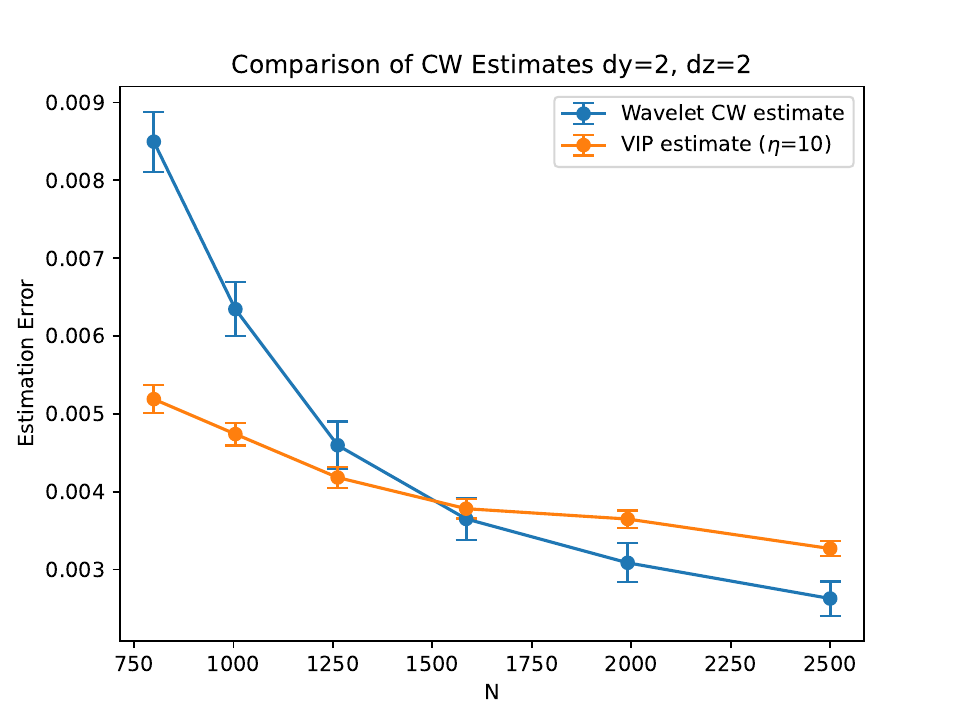}
                \subcaption{Location model}
        \end{minipage}
         \begin{minipage}{0.32\textwidth}
                \centering
                \includegraphics[clip, trim = 0cm 0cm 0cm 0.75cm, width = 1\textwidth]{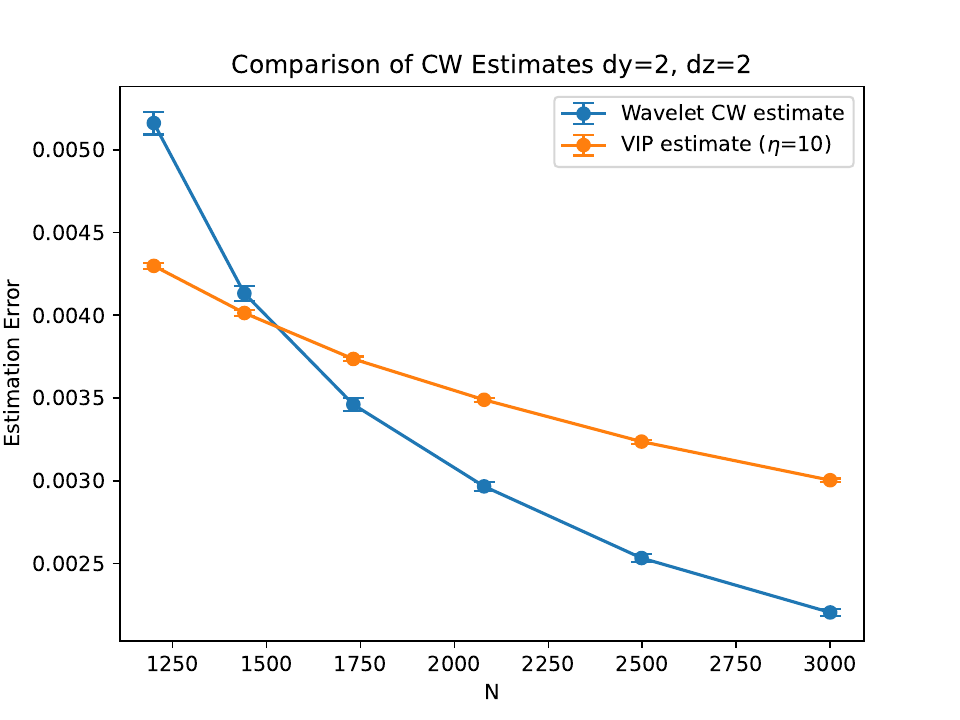}
               \subcaption{Quadratic model} 
        \end{minipage}
         \begin{minipage}{0.32\textwidth}
                \centering
                \includegraphics[clip, trim = 0cm 0cm 0cm 0.75cm, width = 1\textwidth]{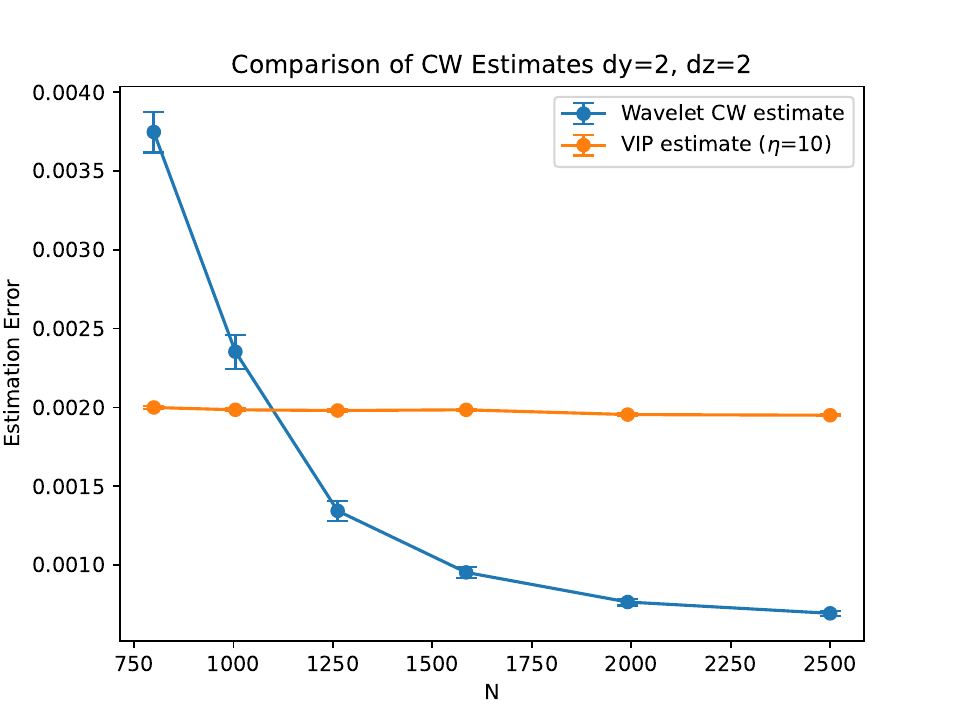}
                \subcaption{Scale model}
        \end{minipage}
        \caption{Plots of estimation error comparison of our method (\textit{Wavelet CW estimate}) and \parencite{lin2025tightening} (\textit{VIP estimate} with $\eta = 10$) with $d_Y=2$, $d_Z=2$. The estimation error is defined by $\EE[|V_{\cc} - \widehat V_{\cc}|]$. The mean error curve and the corresponding 90\% standard-error confidence bands are computed by aggregating results over 300 Monte Carlo repetitions.
        }
        \label{fig:convergence_rateII}
\end{figure*}

\subsection{Simulation Details of Inference}\label{appe:sec:numerical.experiment.simulation.details}

\subsubsection{Data Generation Mechanism} 

In all scenarios,
\[
Y(w)\mid Z=z \sim \mathcal{N}\big(\mu_w(z), \Sigma_w\big), \quad w\in\{0,1\}.
\]

\paragraph{Scenario 1: $d_Y=2,\; d_Z=1$}
\[
\mu_0(z)=
\begin{pmatrix}0.35\\0.55\end{pmatrix}
+
\begin{pmatrix}0.10\\-0.05\end{pmatrix}z,
\]
\[
\mu_1(z)=\mu_0(z)+
\begin{pmatrix}0.12\\-0.08\end{pmatrix}
+
\begin{pmatrix}0.10\\0.02\end{pmatrix}z.
\]
\[
\Sigma_0=
\begin{pmatrix}
0.05^2 & 0\\
0 & 0.03^2
\end{pmatrix},\qquad
\Sigma_1=
\begin{pmatrix}
0.07^2 & 0.01\cdot0.07\cdot0.04\\
0.01\cdot0.07\cdot0.04 & 0.04^2
\end{pmatrix}.
\]

\paragraph{Scenario 2: $d_Y=2,\; d_Z=2$ (default)}
\[
\mu_0(z)=
\begin{pmatrix}0.35\\0.55\end{pmatrix}
+
\begin{pmatrix}
0.20 & -0.10\\
0.05 & 0.15
\end{pmatrix}z,
\]

\[
\mu_1(z)=\mu_0(z)+
\begin{pmatrix}0.12\\-0.08\end{pmatrix}
+
\begin{pmatrix}
0.10 & 0.04\\
0.02 & 0.06
\end{pmatrix}z.
\]
\[
\Sigma_0=
\begin{pmatrix}
0.05^2 & 0\\
0 & 0.03^2
\end{pmatrix},\qquad
\Sigma_1=
\begin{pmatrix}
0.07^2 & 0.01\cdot0.07\cdot0.04\\
0.01\cdot0.07\cdot0.04 & 0.04^2
\end{pmatrix}.
\]

\paragraph{Scenario 3: $d_Y=3,\; d_Z=2$}
\[
\mu_0(z)=
\begin{pmatrix}0.35\\0.55\\0.45\end{pmatrix}
+
\begin{pmatrix}
0.20 & -0.10\\
0.05 & 0.15\\
-0.12 & 0.08
\end{pmatrix}z,
\]
\[
\mu_1(z)=\mu_0(z)+
\begin{pmatrix}0.12\\-0.08\\0.05\end{pmatrix}
+
\begin{pmatrix}
0.10 & 0.04\\
0.02 & 0.06\\
-0.03 & 0.01
\end{pmatrix}z.
\]
\[
\Sigma_0=
\begin{pmatrix}
0.05^2 & 0 & 0\\
0 & 0.03^2 & 0\\
0 & 0 & 0.04^2
\end{pmatrix},\quad
\Sigma_1=
\begin{pmatrix}
0.07^2 & 0.01\cdot0.07\cdot0.04 & 0\\
0.01\cdot0.07\cdot0.04 & 0.04^2 & 0\\
0 & 0 & 0.05^2
\end{pmatrix}.
\]

\subsubsection{Wavelet Algorithm Parameters}

\begin{itemize}
    \item \texttt{wavelet = "db4"}: Daubechies-4 wavelet basis, chosen for compact support and good smoothness properties.

    \item \texttt{mode = "periodization"}: boundary handling rule that treats the support as periodic, avoiding boundary artifacts for data supported on $[0,1]^d$.

    \item \texttt{J\_joint = 4}: wavelet resolution level for estimating the joint density $\widehat p_{Y,Z}(y,z)$. 
    Larger values lead to finer resolution (lower bias, higher variance) and constitute the primary smoothing parameter.

    \item \texttt{J\_z = 6}: wavelet resolution level for estimating the marginal density $\widehat p_Z(z)$. 
    A higher level is used since $Z$ is lower dimensional and easier to estimate accurately.
    
    \item \texttt{threshold = None}: no coefficient thresholding is applied, preserving the plug-in estimator form and simplifying theoretical analysis.

    \item \texttt{nonnegativity = True}: enforces nonnegativity of the reconstructed density, preventing oscillations from producing negative values.

    \item \texttt{renormalize = True}: rescales the estimated density to integrate to one, ensuring it is a valid probability density.

     \item $N_z = 120,\; N_y = 120$: number of grid points used to discretize the supports of $Z$ and $Y$ when numerically evaluating the conditional Wasserstein functional.
\end{itemize}

\subsection{Inference}\label{appe:sec:numerical.experiment.inference}

In \Cref{fig:asymptotic.distribution.dy2dz1}, we illustrate the asymptotic distribution of our estimator 
$\Cwass_2^2(\widehat{\PP}_n^\dagger,\widehat{\QQ}_n^\dagger)$ for $d_Y=2$, $d_Z=1$. 
In \Cref{fig:asymptotic.distribution.dy3dz2}, we present the corresponding asymptotic distribution for $d_Y=3$, $d_Z=2$.
The results are consistent with those in \Cref{sec:empirical.inference}: larger sample sizes decrease the bias and improve the normality.

\begin{figure*}[tbp]
        \centering
        \begin{minipage}{0.32\textwidth}
                \centering
                \includegraphics[clip, trim = 0cm 0cm 0cm 0.75cm, width = 1\textwidth]{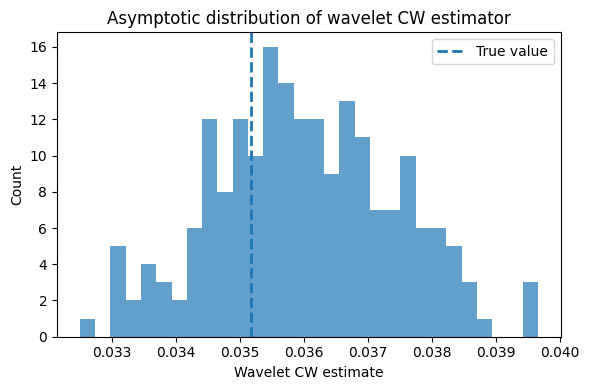}
                \subcaption{$n = 1000$}
        \end{minipage}
         \begin{minipage}{0.32\textwidth}
                \centering
                \includegraphics[clip, trim = 0cm 0cm 0cm 0.75cm, width = 1\textwidth]{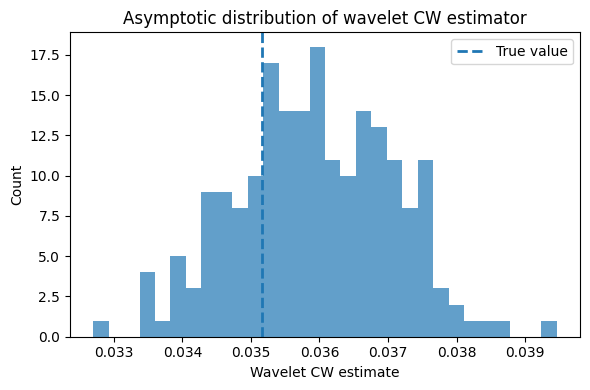}
               \subcaption{$n = 2000$} 
        \end{minipage}
         \begin{minipage}{0.32\textwidth}
                \centering
                \includegraphics[clip, trim = 0cm 0cm 0cm 0.75cm, width = 1\textwidth]{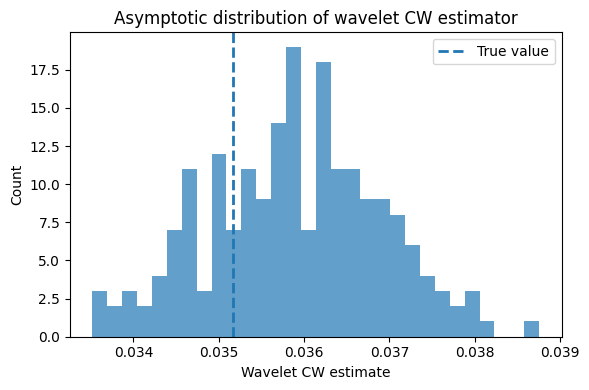}
                \subcaption{$n = 3000$}
        \end{minipage}
        \caption{Histogram of the proposed estimator with the true value marked by a vertical dashed line. 
        Here the default setting $d_Y=2$, $d_Z=1$ is adopted. The results are aggregated over $200$ Monte Carlo repetitions.
        }
        \label{fig:asymptotic.distribution.dy2dz1}
\end{figure*}

\begin{figure*}[tbp]
        \centering
        \begin{minipage}{0.32\textwidth}
                \centering
                \includegraphics[clip, trim = 0cm 0cm 0cm 0.75cm, width = 1\textwidth]{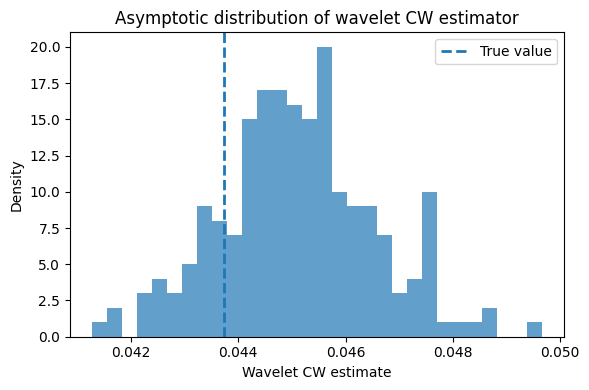}
                \subcaption{$n = 4000$}
        \end{minipage}
         \begin{minipage}{0.32\textwidth}
                \centering
                \includegraphics[clip, trim = 0cm 0cm 0cm 0.75cm, width = 1\textwidth]{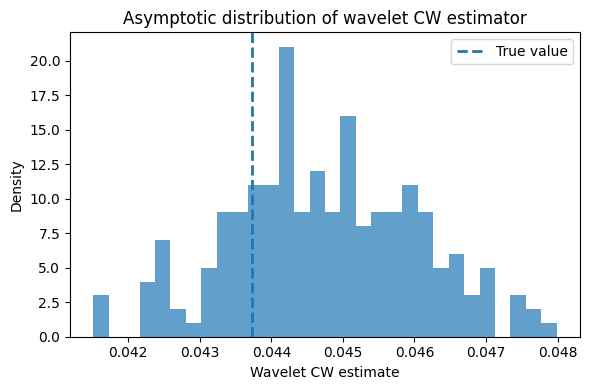}
               \subcaption{$n = 5000$} 
        \end{minipage}
         \begin{minipage}{0.32\textwidth}
                \centering
                \includegraphics[clip, trim = 0cm 0cm 0cm 0.75cm, width = 1\textwidth]{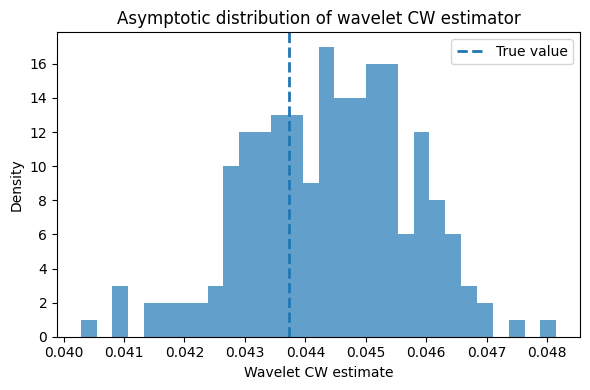}
                \subcaption{$n = 6000$}
        \end{minipage}
        \caption{Histogram of the proposed estimator with the true value marked by a vertical dashed line. 
        Here the default setting $d_Y=3$, $d_Z=2$ is adopted. The results are aggregated over $200$ Monte Carlo repetitions.
        }
        \label{fig:asymptotic.distribution.dy3dz2}
\end{figure*}


\end{document}

%% file: sty.tex
\def\EE{\mathbb{E}}

\def\RR{\mathbb{R}}

\def\calB{\mathcal{B}}
\def\calC{\mathcal{C}}

\def\calI{\mathcal{I}}

\def\calK{\mathcal{K}}

\def\calN{\mathcal{N}}

\def\calP{\mathcal{P}}

\def\calS{\mathcal{S}}

\def\calX{\mathcal{X}}
\def\calY{\mathcal{Y}}
\def\calZ{\mathcal{Z}}

\def\1{\mathbbm{1}}
\def\var{\mathsf{Var}}

\usepackage{booktabs}
\usepackage{adjustbox}
\usepackage{multirow}
\usepackage{listings}

\usepackage{tikz}

\usepackage{xspace}

\definecolor{myblue}{rgb}{.8, .8, 1}
\definecolor{mathblue}{rgb}{0.2472, 0.24, 0.6} 
\definecolor{mathred}{rgb}{0.6, 0.24, 0.442893}
\definecolor{mathyellow}{rgb}{0.6, 0.547014, 0.24}

\usepackage{enumitem}